%% file: Pie_Flex_Shells_2022_R1.tex
\documentclass{amsart}
\usepackage{moreverb,url}
\usepackage{amssymb,amsmath,amsfonts,mathrsfs}
\usepackage{bm}
\usepackage{float}
\usepackage{verbatim}
\usepackage[square,numbers]{natbib}
\usepackage{stmaryrd}
\usepackage{fancyhdr}
\usepackage{syntax,etoolbox}
\usepackage{graphicx}
\usepackage{xcolor}
\usepackage{hyperref}
\hypersetup{colorlinks=true,allcolors=blue}
\usepackage[square,numbers]{natbib}
\usepackage[left=2.5cm, right=2.5cm, bottom=3cm]{geometry}
\usepackage{lineno}
\newtheorem{theorem}{Theorem}[section]
\newtheorem{lemma}{Lemma}[section]

\newtheorem{corollary}[theorem]{Corollary}

\newtheorem{innercustomgeneric}{\customgenericname}
\providecommand{\customgenericname}{}
\newcommand{\newcustomproblem}[2]{%
	\newenvironment{#1}[1]
	{%
		\renewcommand\customgenericname{#2}%
		\renewcommand\theinnercustomgeneric{##1}%
		\innercustomgeneric
	}
	{\endinnercustomgeneric}
}

\newcustomproblem{customprob}{Problem}

\newcommand*{\bqed}{\hfill\ensuremath{\blacksquare}}%

\def\dd{\, \mathrm{d}}

\begin{document}
	
	\today
	
	%%
	%% The title of the paper goes here.  Edit to your title.
	%%
	
	\title[An obstacle problem for flexural shells]{Asymptotic analysis of linearly elastic flexural shells subjected to an obstacle in absence of friction}
	
	%%
	%% Now edit the following to give your name and address:
	%% 
	
	\author{Paolo Piersanti}
	\address{Department of Mathematics and Institute for Scientific Computing and Applied Mathematics, Indiana University Bloomington, 729 East Third Street, Bloomington, Indiana, USA}
	\email{ppiersan@iu.edu}

\begin{abstract}
In this paper we identify a set of two-dimensional variational inequalities governing the displacement of a linearly elastic flexural shell subjected to a confinement condition, expressing that all the points of the admissible deformed configurations remain in a given half-space. The action of friction is neglected.

\end{abstract}

\maketitle

\input{paper.tex}

\bibliographystyle{abbrvnat} %unsrt	
\bibliography{references.bib}	

\end{document}

%% file: paper.tex
\section{Introduction}
\label{Sec:0}

Unilateral contact problems arise in many fields such as medicine, engineering, biology and material science. For instance, the description of the motion inside the human heart of the three Aorta valves, which can be regarded as linearly elastic shells, is governed by a mathematical model built up in a way such that each valve remains confined in a certain portion of space without penetrating or being penetrated by the other two valves. In this direction we cite the recent references~\cite{Quarteroni2021-3}, \cite{Quarteroni2021-2} and~\cite{Quarteroni2021-1}.

The displacement of a linearly elastic shell is modeled, in general, via the classical equations of three-dimensional linearized elasticity (cf., e.g., \cite{Ciarlet1988}). The intrinsic complexity of this model, however, prevents certain situations from being amenably studied like, for instance, when the shell is non-homogeneous and anisotropic (cf., e.g., \cite{CaillerieSanchez1995a} and~\cite{CaillerieSanchez1995b}), or its thickness varies periodically (cf., e.g., \cite{TelegaLewinski1998a} and~\cite{TelegaLewinski1998b}). It might thus be useful to perform a \emph{dimension reduction} in order to obtain approximate models which are more amenable to analyze and to implement numerically.

The identification of two-dimensional limit models for time-independent linearly elastic shells was extensively treated by Ciarlet and his associates in the seminal papers~\cite{CiaDes1979, CiaLods1996a,CiaLods1996b,CiaLods1996c,CiaLods1996d,CiaLodsMia1996} for the purpose of justifying Koiter's model in the case where no obstacles are taken into account. We also mention the papers~\cite{Shen2018,Shen2019,Shen2020}, which are about the numerical computation of the solution of the obstacle-free linear Koiter's model in the time-dependent case.
For the justification of Koiter's model in the time-dependent case where the action of temperature is considered, we refer the reader to~\cite{Pie-2021}.
In the aforementioned papers no confinement conditions were imposed. The recent papers~\cite{PWDT2D} and~\cite{PWDT3D} provide examples of variational-inequalities-based models arising in biology. The paper~\cite{PieTem2023}, instead, treats the modeling and analysis of the melting of a shallow ice sheet by means of a set of doubly nonlinear parabolic variational inequalities.

In the recent papers~\cite{CiaPie2018b,CiaPie2018bCR,CiaMarPie2018b,CiaMarPie2018}, Ciarlet and his associates fully justified Koiter's model in the case where the linearly elastic shell under consideration is an elliptic membrane shell subject to the aforementioned confinement condition.

The confinement condition we are here considering considerably departs from the Signorini condition usually considered in the existing literature. Indeed, the Signorini contact condition states that only the ``lower face'' of the shell is required to remain above the ``horizontal'' plane (cf., e.g. Chapter~63 of~\cite{Zeidler1986}). Such a confinement condition renders the asymptotic analysis considerably more difficult as the constraint now bears on a vector field, the displacement vector field of the reference configuration, instead of on only a single component of this field.

The recovery of a set of two-dimensional variational inequalities as a result of a rigorous asymptotic analysis conducted on a \emph{ad hoc} three-dimensional model based on the classical equations of three-dimensional linearized elasticity in the case where the linearly elastic shell under consideration is a flexural shell has only been addressed, to our best knowledge, by L\'{e}ger and Miara in the paper~\cite{LegMia2018} in a very amenable geometrical framework. In the paper~\cite{LegMia2018} the authors do not consider the action of friction. The method proposed in the paper~\cite{LegMia2018}, which makes use of a very lengthy argument exploiting the properties of cones, is suitable for treating the case where only one linearly elastic flexural shell is considered. In the context of multi-physics multi-scale problems, like for instance those considered by Quarteroni and his associates in~\cite{Quarteroni2021-3}, \cite{Quarteroni2021-2} and~\cite{Quarteroni2021-1}, it is not however clear whether the approach proposed by L\'{e}ger and Miara is suitable. The purpose of this paper is to remedy this situation by presenting a new method, based on the rigorous asymptotic analysis technique developed by Ciarlet, Lods and Miara~\cite{CiaLodsMia1996} and the properties of the penalty method for constrained optimization problem (cf., e.g., \cite{PGCNLAO}), for recovering the same set of two-dimensional variational inequalities for linearly elastic flexural shells that were recovered in the recent papers~\cite{CiaPie2018b,CiaPie2018bCR} in a more general geometrical framework. It is also worth mentioning that the method we present in this paper makes use of considerably less lengthy computations than those in~\cite{LegMia2018}. In particular, the approach based on the penalty method we are going to implement will save us the effort of constructing a suitable vector field in the instance of recovering the sought \emph{two-dimensional limit model}. We refer the reader to the references~\cite{KikuchiOden1988,PieMMS2020,Scholz1986,ShenPiePie2020}.

More precisely, in this paper we recover, via a rigorous asymptotic analysis as the thickness approaches zero over a \emph{ad hoc} three-dimensional model (\emph{three-dimensional} in the sense that it is defined over a three-dimensional subset of $\mathbb{R}^3$), a set of two-dimensional (\emph{two-dimensional} in the sense that it is defined over a two-dimensional subset of $\mathbb{R}^2$) variational inequalities governing the displacement of a linearly elastic flexural shell subject to remain confined in a half-space in absence of friction. The problem under consideration is an \emph{obstacle problem}.

The paper is divided into five sections (including this one). In section~\ref{Sec:1} we recall some background and notation. In section~\ref{Sec:2} we recall the formulation and the properties of a three-dimensional obstacle problem for ``general'' linearly elastic shell. In section~\ref{Sec:3} we specialize the formulation presented in the previous section to the case where the linearly elastic shell under consideration is a flexural shell, and then we scale the three-dimensional obstacle problem in a way such that the integration domain becomes independent of the thickness parameter. The penalized version of the three-dimensional obstacle problem is introduced at the end of this section.
Finally, in section~\ref{Sec:4}, a rigorous asymptotic analysis is carried out and the sought set of two-dimensional variational inequalities is recovered.

\section{Geometrical preliminaries} \label{Sec:1}

For details about the classical notions of differential geometry used in this section and the next one, see, e.g.~\cite{Ciarlet2000} or \cite{Ciarlet2005}.

Greek indices, except $\varepsilon$, take their values in the set $\{1,2\}$, while Latin indices, except when they are used for indexing sequences, take their values in the set $\{1,2,3\}$, and the summation convention with respect to repeated indices is systematically used in conjunction with these two rules. The notation $\mathbb{E}^3$ designates the three-dimensional Euclidean space whose origin is denoted by $O$; the Euclidean inner product and the vector product of $\bm{u}, \bm{v} \in \mathbb{E}^3$ are denoted $\bm{u} \cdot \bm{v}$ and $\bm{u} \times \bm{v}$; the Euclidean norm of $\bm{u} \in \mathbb{E}^3$ is denoted $\left|\bm{u} \right|$. The notation $\delta^j_i$ designates the Kronecker symbol.

Given an open subset $\Omega$ of $\mathbb{R}^n$, notations such as $L^2(\Omega)$, $H^1(\Omega)$, or $H^2 (\Omega)$, designate the usual Lebesgue and Sobolev spaces, and the notation $\mathcal{D} (\Omega)$ designates the space of all functions that are infinitely differentiable over $\Omega$ and have compact supports in $\Omega$. The notation $\left\| \cdot \right\|_X$ designates the norm in a normed vector space $X$. Spaces of vector-valued functions are denoted with boldface letters.

The positive and negative parts of a function $f:\Omega \to \mathbb{R}$ are respectively denoted by:
$$
f^{+}(x):=\max\{f(x),0\}\quad\textup{ and }\quad f^{-}(x):=-\min\{f(x),0\} \quad x \in \Omega.
$$

The boundary $\Gamma$ of an open subset $\Omega$ in $\mathbb{R}^n$ is said to be Lipschitz-continuous if the following conditions are satisfied (cf., e.g., Section~1.18 of~\cite{PGCLNFAA}): Given an integer $s\ge 1$, there exist constants $\alpha_1>0$ and $L>0$, and a finite number of local coordinate systems, with coordinates $\bm{\phi}'_r=(\phi_1^r, \dots, \phi_{n-1}^r) \in \mathbb{R}^{n-1}$ and $\phi_r=\phi_n^r$, sets $\tilde{\omega}_r:=\{\bm{\phi}_r \in\mathbb{R}^{n-1}; |\bm{\phi}_r|<\alpha_1\}$, $1 \le r \le s$, and corresponding functions
$$
\tilde{\theta}_r:\tilde{\omega}_r\to\mathbb{R},
$$
such that
$$
\Gamma=\bigcup_{r=1}^s \{(\bm{\phi}'_r,\phi_r); \bm{\phi}'_r \in \tilde{\omega}_r \textup{ and }\phi_r=\tilde{\theta}_r(\bm{\phi}'_r)\},
$$
and 
$$
|\tilde{\theta}_r(\bm{\phi}'_r)-\tilde{\theta}_r(\bm{\upsilon}'_r)|\le L |\bm{\phi}'_r-\bm{\upsilon}'_r|, \textup{ for all }\bm{\phi}'_r, \bm{\upsilon}'_r \in \tilde{\omega}_r, \textup{ and all }1\le r\le s.
$$

We observe that the second last formula takes into account overlapping local charts, while the last set of inequalities express the Lipschitz continuity of the mappings $\tilde{\theta}_r$.

An open set $\Omega$ is said to be locally on the same side of its boundary $\Gamma$ if, in addition, there exists a constant $\alpha_2>0$ such that
\begin{align*}
	\{(\bm{\phi}'_r,\phi_r);\bm{\phi}'_r \in\tilde{\omega}_r \textup{ and }\tilde{\theta}_r(\bm{\phi}'_r) < \phi_r < \tilde{\theta}_r(\bm{\phi}'_r)+\alpha_2\}&\subset \Omega,\quad\textup{ for all } 1\le r\le s,\\
	\{(\bm{\phi}'_r,\phi_r);\bm{\phi}'_r \in\tilde{\omega}_r \textup{ and }\tilde{\theta}_r(\bm{\phi}'_r)-\alpha_2 < \phi_r < \tilde{\theta}_r(\bm{\phi}'_r)\}&\subset \mathbb{R}^n\setminus\overline{\Omega},\quad\textup{ for all } 1\le r\le s.
\end{align*}

Following~\cite{PGCLNFAA}, a \emph{domain in} $\mathbb{R}^n$ is considered as a bounded Lipschitz domain, namely, a bounded and connected open subset $\Omega$ of $\mathbb{R}^n$, whose boundary $\partial \Omega$ is Lipschitz-continuous, the set $\Omega$ being locally on a single side of $\partial \Omega$.

Let $\omega$ be a domain in $\mathbb{R}^2$, let $y = (y_\alpha)$ denote a generic point in $\overline{\omega}$, and let $\partial_\alpha := \partial / \partial y_\alpha$ and $\partial_{\alpha \beta} := \partial^2/\partial y_\alpha \partial y_\beta$. A mapping $\bm{\theta} \in \mathcal{C}^1(\overline{\omega}; \mathbb{E}^3)$ is an \emph{immersion} if the two vectors
$$
\bm{a}_\alpha(y) := \partial_\alpha \bm{\theta}(y)
$$
are linearly independent at each point $y \in \overline{\omega}$. Then the image $\bm{\theta} (\overline{\omega})$ of the set $\overline{\omega}$ under the mapping $\bm{\theta}$ is a \emph{surface in} $\mathbb{E}^3$, equipped with $y_1, y_2$ as its \emph{curvilinear coordinates}. Given any point $y\in \overline{\omega}$, the vectors $\bm{a}_\alpha (y)$ span the \emph{tangent plane} to the surface $\bm{\theta} (\overline{\omega})$ at the point $\bm{\theta} (y)$, the unit vector
$$
\bm{a}_3(y) := \frac{\bm{a}_1(y) \times \bm{a}_2 (y)}{|\bm{a}_1(y) \times \bm{a}_2(y)|}
$$
is normal to $\bm{\theta} (\overline{\omega})$ at $\bm{\theta} (y)$, the three vectors $\bm{a}_i(y)$ form the \emph{covariant} basis at $\bm{\theta}(y)$, and the three vectors $\bm{a}^j(y)$ defined by the relations
$$
\bm{a}^j(y) \cdot \bm{a}_i(y) = \delta^j_i
$$
form the \emph{contravariant} basis at $\bm{\theta} (y)$; note that the vectors $\bm{a}^\beta (y)$ also span the tangent plane to $\bm{\theta} (\overline{\omega})$ at $\bm{\theta} (y)$ and that $\bm{a}^3(y) = \bm{a}_3 (y)$.

The \emph{first fundamental form} of the surface $\bm{\theta} (\overline{\omega})$ is then defined by means of its \emph{covariant components}
\[
a_{\alpha \beta} := \bm{a}_\alpha \cdot \bm{a}_\beta = a_{\beta \alpha} \in \mathcal{C}^0 (\overline{\omega}),
\]
or by means of its \emph{contravariant components}
\[
a^{\alpha \beta}:= \bm{a}^\alpha \cdot \bm{a}^\beta = a^{\beta \alpha}\in \mathcal{C}^0(\overline{\omega}).
\]
Note that the symmetric matrix field $(a^{\alpha \beta})$ is then the inverse of the matrix field $(a_{\alpha \beta})$, that $\bm{a}^\beta = a^{\alpha \beta}\bm{a}_\alpha$ and $\bm{a}_\alpha = a_{\alpha \beta} \bm{a}^\beta$, and that the \emph{area element} along $\bm{\theta} (\overline{\omega})$ is given at each point $\bm{\theta} (y), \, y \in \overline{\omega}$, by $\sqrt{a(y)}\dd y$, where
\[
a := \det (a_{\alpha \beta}) \in \mathcal{C}^0 (\overline{\omega}).
\]

Given an immersion $\bm{\theta} \in \mathcal{C}^2(\overline{\omega}; \mathbb{E}^3)$, the \emph{second fundamental form} of the surface $\bm{\theta} (\overline{\omega})$ is defined by means of its \emph{covariant components}
\[
b_{\alpha \beta}:= \partial_\alpha \bm{a}_\beta \cdot \bm{a}_3 = - \bm{a}_\beta \cdot \partial_\alpha \bm{a}_3 = b_{\beta \alpha} \in \mathcal{C}^0(\overline{\omega}),
\]
or by means of its \emph{mixed components}
\[
b^\beta_\alpha := a^{\beta \sigma} b_{\alpha \sigma} \in \mathcal{C}^0(\overline{\omega}),
\]
and the \emph{Christoffel symbols} associated with the immersion $\bm{\theta}$ are defined by
\[
\Gamma^\sigma_{\alpha \beta}:= \partial_\alpha \bm{a}_\beta \cdot \bm{a}^\sigma = \Gamma^\sigma_{\beta \alpha} \in \mathcal{C}^0 (\overline{\omega}).
\]

%The \emph{Gaussian curvature} at each point $\bm{\theta} (y) , \, y \in \overline{\omega}$, of the surface $\bm{\theta}(\overline{\omega})$ is defined by
%\[
%K(y) := \frac{\det(b_{\alpha \beta}(y))}{\det(a_{\alpha \beta}(y))} = \det \left(b^\beta_\alpha(y)\right)
%\]
%(the denominator in the above relation does not vanish since $\bm{\theta}$ is assumed to be an immersion). Note that the Gaussian curvature $K(y)$ at the point $\bm{\theta}(y)$ is also equal to the inverse of the product of the two principal radii of curvature at this point.

Given an immersion
$\bm{\theta} \in \mathcal{C}^2 (\overline{\omega}; \mathbb{E}^3)$ and a
vector field $\bm{\eta} = (\eta_i) \in \mathcal{C}^1(\overline{\omega};
\mathbb{R}^3)$, the vector field
\[
\tilde{\bm{\eta}} := \eta_i \bm{a}^i
\]
can be viewed as a \emph{displacement field of the surface} $\bm{\theta} (\overline{\omega})$, thus defined by means of its \emph{covariant components} $\eta_i$ over the vectors $\bm{a}^i$ of the contravariant bases along the surface. If the norms $\left\| \eta_i \right\|_{\mathcal{C}^1(\overline{\omega})}$ are small enough, the mapping $(\bm{\theta} + \eta_i \bm{a}^i) \in \mathcal{C}^1(\overline{\omega}; \mathbb{E}^3)$ is also an immersion, so that the set $(\bm{\theta} + \eta_i \bm{a}^i) (\overline{\omega})$ is also a surface in $\mathbb{E}^3$, equipped with the same curvilinear coordinates as those of the surface $\bm{\theta} (\overline{\omega})$, called the \emph{deformed surface} corresponding to the displacement field $\tilde{\bm{\eta}} = \eta_i \bm{a}^i$. One can then define the first fundamental form of the deformed surface by means of its covariant components
\begin{align*}
  a_{\alpha \beta} (\bm{\eta}) :=& (\bm{a}_\alpha + \partial_\alpha \tilde{\bm{\eta}}) \cdot (\bm{a}_\beta + \partial_\beta \tilde{\bm{\eta}}) \\
  =& a_{\alpha \beta} + \bm{a}_\alpha \cdot \partial_\beta \tilde{\bm{\eta}} + \partial_\alpha \tilde{\bm{\eta}} \cdot \bm{a}_\beta + \partial_\alpha \tilde{\bm{\eta}} \cdot \partial_\beta \tilde{\bm{\eta}}.
\end{align*}

The \emph{linear part with  respect to} $\tilde{\bm{\eta}}$ in the difference $\dfrac12 (a_{\alpha \beta}(\bm{\eta}) - a_{\alpha \beta})$ is called the \emph{linearized change of metric tensor} associated with the displacement field $\eta_i \bm{a}^i$, the covariant components of which are thus defined by
\[
  \gamma_{\alpha \beta}(\bm{\eta}) := \dfrac12 \left( \bm{a}_\alpha \cdot \partial_\beta \tilde{\bm{\eta}} + \partial_\alpha \tilde{\bm{\eta}} \cdot \bm{a}_\beta \right) = \frac12 (\partial_\beta \eta_\alpha + \partial_\alpha \eta_\beta ) - \Gamma^\sigma_{\alpha \beta} \eta_\sigma - b_{\alpha \beta} \eta_3 = \gamma_{\beta \alpha} (\bm{\eta}).
\]

The \emph{linear part with  respect to} $\tilde{\bm{\eta}}$ in the difference $\dfrac12 (b_{\alpha \beta}(\bm{\eta}) - b_{\alpha \beta})$ is called the \emph{linearized change of curvature tensor} associated with the displacement field $\eta_i \bm{a}^i$, the covariant components of which are thus defined by
\begin{align*}
\rho_{\alpha \beta}(\bm{\eta}) &:= (\partial_{\alpha \beta}\tilde{\bm{\eta}}-\Gamma_{\alpha\beta}^\sigma \partial_\sigma \tilde{\bm{\eta}})\cdot \bm{a}_3 = \rho_{\beta \alpha}(\bm{\eta})\\
&= \partial_{\alpha \beta}\eta_3 -\Gamma_{\alpha\beta}^\sigma \partial_\sigma\eta_3 -b_\alpha^\sigma b_{\sigma\beta}\eta_3
+b_\alpha^\sigma(\partial_\beta \eta_\sigma-\Gamma_{\beta\sigma}^\tau\eta_\tau)+b_\beta^\tau(\partial_\alpha\eta_\tau-\Gamma_{\alpha\tau}^\sigma\eta_\sigma)
+(\partial_\alpha b_\beta^\tau+\Gamma_{\alpha\sigma}^\tau b_\beta^\sigma-\Gamma_{\alpha\beta}^\sigma b_\sigma^\tau)\eta_\tau.
\end{align*}
It turns out that, when a \emph{generic surface} is subjected to a displacement field $\eta_i \bm{a}^i$ whose \emph{tangential covariant components $\eta_\alpha$ vanish on a non-zero length portion of boundary of the domain $\omega$}, denoted $\gamma_0$ in the statement of the next result, the following inequality holds (this inequality plays an essential role in our convergence analysis; cf.\ the proof of Theorem \ref{t:5}). Note that the components of the displacement fields, linearized change of metric tensor and linearized change of curvature tensor appearing in the next theorem are no longer assumed to be continuously differentiable functions; they are instead to be understood in a generalized sense, since they now belong to ad hoc Lebesgue or Sobolev spaces.
Throughout the paper the symbol $\partial_{\nu}$ denotes the \emph{outer unit normal derivative operator along the boundary $\gamma$.}

\begin{theorem} \label{t:1}
 Let $\omega$ be a domain in $\mathbb{R}^2$ and let an immersion $\bm{\theta} \in \mathcal{C}^3 (\overline{\omega}; \mathbb{E}^3)$ be given. Define the space
\begin{align*}
\bm{V}_K (\omega)&:= \{\bm{\eta}=(\eta_i) \in H^1(\omega)\times H^1(\omega)\times H^2(\omega);\eta_i=\partial_{\nu}\eta_3=0 \textup{ on }\gamma_0\}.
\end{align*}
Then there exists a constant $c_0=c_0(\omega,\gamma_0,\bm{\theta})>0$ such that
\[
\left\{ \sum_\alpha \left\| \eta_\alpha \right\|^2_{H^1(\omega)} + \left\| \eta_3 \right\|^2_{H^2(\omega)}\right\}^{1/2} \le c_0 \left\{ \sum_{\alpha, \beta} \left\| \gamma_{\alpha \beta}(\bm{\eta}) \right\|_{L^2(\omega)}^2+\sum_{\alpha, \beta} \left\| \rho_{\alpha \beta}(\bm{\eta}) \right\|_{L^2(\omega)}^2\right\}^{1/2}
\]
for all $\bm{\eta} = (\eta_i) \in \bm{V}_K (\omega)$.
\qed
\end{theorem}

The above inequality, which is due to \cite{BerCia1976} and was later on improved by \cite{BerCiaMia1994} (see also Theorem~2.6-4 of~\cite{Ciarlet2000}),  constitutes an example of a \emph{Korn inequality on a general surface}, in the sense that it provides an estimate of an appropriate norm of a displacement field defined on a surface in terms of an appropriate norm of a specific ``measure of strain'' (here, the linearized change of metric tensor and the linearized change of curvature tensor) corresponding to the displacement field considered.

\section{The three-dimensional obstacle problem for a ``general'' linearly elastic shell} \label{Sec:2}

Let $\omega$ be a domain in $\mathbb{R}^2$, let $\gamma:= \partial \omega$, and let $\gamma_0$ be a non-empty relatively open subset of $\gamma$. For each $\varepsilon > 0$, we define the sets
\[
\Omega^\varepsilon = \omega \times \left] - \varepsilon , \varepsilon \right[ \quad \textup{ and } \quad \Gamma^\varepsilon_0 := \gamma_0 \times \left[ - \varepsilon , \varepsilon \right],
\]
we let $x^\varepsilon = (x^\varepsilon_i)$ designate a generic point in the set $\overline{\Omega^\varepsilon}$, and we let $\partial^\varepsilon_i := \partial / \partial x^\varepsilon_i$. Hence we also have $x^\varepsilon_\alpha = y_\alpha$ and $\partial^\varepsilon_\alpha = \partial_\alpha$.

Given an immersion $\bm{\theta} \in \mathcal{C}^3(\overline{\omega}; \mathbb{E}^3)$ and $\varepsilon > 0$, consider a \emph{shell} with \emph{middle surface} $\bm{\theta} (\overline{\omega})$ and with \emph{constant thickness} $2 \varepsilon$. This means that the \emph{reference configuration} of the shell is the set $\bm{\Theta} (\overline{\Omega^\varepsilon})$, where the mapping $\bm{\Theta} : \overline{\Omega^\varepsilon} \to \mathbb{E}^3$ is defined by
\[
\bm{\Theta} (x^\varepsilon) := \bm{\theta} (y) + x^\varepsilon_3 \bm{a}^3(y) \text{ at each point } x^\varepsilon = (y, x^\varepsilon_3) \in \overline{\Omega^\varepsilon}.
\]

One can then show (cf., e.g., Theorem 3.1-1 of~\cite{Ciarlet2000}) that, if $\varepsilon > 0$ is small enough, such a mapping $\bm{\Theta} \in \mathcal{C}^2(\overline{\Omega^\varepsilon}; \mathbb{E}^3)$ is an \emph{immersion}, in the sense that the three vectors
\[
\bm{g}^\varepsilon_i (x^\varepsilon) := \partial^\varepsilon_i \bm{\Theta} (x^\varepsilon),
\]
are linearly independent at each point $x^\varepsilon \in \overline{\Omega^\varepsilon}$; these vectors  then constitute the \emph{covariant basis} at the point $\bm{\Theta} (x^\varepsilon)$, while the three vectors $\bm{g}^{j, \varepsilon} (x^\varepsilon)$ defined by the relations
\[
\bm{g}^{j, \varepsilon} (x^\varepsilon) \cdot \bm{g}^\varepsilon_i (x^\varepsilon) = \delta^j_i,
\]
constitute the \emph{contravariant basis} at the same point. It will be implicitly assumed in the sequel that $\varepsilon > 0$ \emph{is small enough so that $\bm{\Theta}:\overline{\Omega^\varepsilon} \to \mathbb{E}^3$} is an \emph{immersion}.

One then defines the \emph{metric tensor associated with the immersion} $\bm{\Theta}$ by means of its \emph{covariant components}
\[
g^\varepsilon_{ij} := \bm{g}^\varepsilon_i \cdot \bm{g}^\varepsilon_j \in \mathcal{C}^1(\overline{\Omega^\varepsilon}),
\]
or by means of its \emph{contravariant components}
\[
g^{ij, \varepsilon} := \bm{g}^{i, \varepsilon} \cdot \bm{g}^{i,\varepsilon} \in \mathcal{C}^1(\overline{\Omega^\varepsilon}).
\]

Note that the symmetric matrix field $(g^{ij, \varepsilon})$ is then the inverse of the matrix field $(g^\varepsilon_{ij})$, that $\bm{g}^{j, \varepsilon} = g^{ij, \varepsilon} \bm{g}^\varepsilon_i$ and $g^\varepsilon_i = g^\varepsilon_{ij} \bm{g}^{j, \varepsilon}$, and that the \emph{volume element} in $\bm{\Theta} (\overline{\Omega^\varepsilon})$ is given at each point $\bm{\Theta} (x^\varepsilon)$, $x^\varepsilon \in \overline{\Omega^\varepsilon}$, by $\sqrt{g^\varepsilon (x^\varepsilon)} \dd x^\varepsilon$, where
\[
g^\varepsilon := \det (g^\varepsilon_{ij}) \in \mathcal{C}^1(\overline{\Omega^\varepsilon}).
\]

One also defines the \emph{Christoffel symbols} associated with the immersion $\bm{\Theta}$ by
\[
\Gamma^{p, \varepsilon}_{ij}:= \partial_i \bm{g}^\varepsilon_j \cdot \bm{g}^{p, \varepsilon} = \Gamma^{p, \varepsilon}_{ji} \in \mathcal{C}^0(\overline{\Omega^\varepsilon}).
\]
Note that $\Gamma^{3,\varepsilon}_{\alpha 3} = \Gamma^{p, \varepsilon}_{33} = 0$.

Given a vector field $\bm{v}^\varepsilon = (v^\varepsilon_i) \in \mathcal{C}^1 (\overline{\Omega^\varepsilon}; \mathbb{R}^3)$, the associated vector field
\[
\tilde{\bm{v}}^\varepsilon := v^\varepsilon_i \bm{g}^{i, \varepsilon},
\]
can be viewed as a \emph{displacement field} of the reference configuration $\bm{\Theta} (\overline{\Omega^\varepsilon})$ of the shell, thus defined by means of its \emph{covariant components} $v^ \varepsilon_i$ over the vectors $\bm{g}^{i, \varepsilon}$ of the contravariant bases in the reference configuration.

If the norms $\left\| v^\varepsilon_i \right\|_{\mathcal{C}^1 (\overline{\Omega^\varepsilon})}$ are small enough, the mapping $(\bm{\Theta} + v^\varepsilon_i \bm{g}^{i, \varepsilon})$ is also an immersion, so that one can also define the metric tensor of the \emph{deformed configuration} $(\bm{\Theta} + v^\varepsilon_i \bm{g}^{i, \varepsilon})(\overline{\Omega^\varepsilon})$ by means of its covariant components
\begin{align*}
  g^\varepsilon_{ij} (v^\varepsilon) := (\bm{g}^\varepsilon_i + \partial^\varepsilon_i \tilde{\bm{v}}^\varepsilon ) \cdot (\bm{g}^\varepsilon_j + \partial^\varepsilon_j \tilde{\bm{v}}^\varepsilon)
  = g^\varepsilon_{ij} + \bm{g}^\varepsilon_i \cdot \partial_j \tilde{\bm{v}}^\varepsilon + \partial^\varepsilon_i \tilde{\bm{v}}^\varepsilon \cdot \bm{g}^\varepsilon_j + \partial_i \tilde{\bm{v}}^\varepsilon \cdot \partial_j \tilde{\bm{v}}^\varepsilon.
\end{align*}

The linear part with respect to $\tilde{\bm{v}}^\varepsilon$ in the difference $\dfrac12(g^\varepsilon_{ij} (\bm{v}^\varepsilon) - g^\varepsilon_{ij})$ is then called the \emph{linearized strain tensor} associated with the displacement field $v^\varepsilon_i \bm{g}^{i, \varepsilon}$, the covariant components of which are thus defined by
\[
e^\varepsilon_{i\|j} (\bm{v}^\varepsilon) := \frac12 \left( \bm{g}^\varepsilon_i \cdot \partial_j^\varepsilon \tilde{\bm{v}}^\varepsilon + \partial^\varepsilon_i \tilde{\bm{v}}^\varepsilon \cdot \bm{g}^\varepsilon_j \right) = \frac12 (\partial^\varepsilon_j v^\varepsilon_i + \partial^\varepsilon_i v^\varepsilon_j) - \Gamma^{p, \varepsilon}_{ij} v^\varepsilon_p = e_{j\|i}^\varepsilon (\bm{v}^\varepsilon).
\]
The functions $e^\varepsilon_{i\|j} (\bm{v}^\varepsilon)$ are called the \emph{linearized strains in curvilinear coordinates} associated with the displacement field $v^\varepsilon_i \bm{g}^{i, \varepsilon}$.

We assume throughout this paper that, for each $\varepsilon > 0$, the reference configuration $\bm{\Theta} (\overline{\Omega^\varepsilon})$ of the shell is a \emph{natural state} (i.e., stress-free) and that the material constituting the shell is \emph{homogeneous}, \emph{isotropic}, and \emph{linearly elastic}. The behavior of such an elastic material is thus entirely governed by its two \emph{Lam\'{e} constants} $\lambda \ge 0$ and $\mu > 0$ (for details, see, e.g., Section~3.8 of~\cite{Ciarlet1988}).

We will also assume that the shell is subjected to \emph{applied body forces} whose density per unit volume is defined by means of its covariant components $f^{i, \varepsilon} \in L^2(\Omega^\varepsilon)$, and to a \emph{homogeneous boundary condition of place} along the portion $\Gamma^\varepsilon_0$ of its lateral face (i.e., the displacement vanishes on $\Gamma^\varepsilon_0$).

For what concerns surface traction forces, the mathematical models characterized by the confinement condition considered in this paper (confinement condition which is also considered in~\cite{LegMia2008} in a more amenable geometrical framework) do not take any surface traction forces into account. 
Indeed, there could be no surface traction forces applied to the portion of the three-dimensional shell boundary that engages contact with the obstacle. 

The confinement condition considered in this paper is more suitable in the context of multi-scales multi-bodies problems like, for instance, the study of the motion of the human heart valves, conducted by Quarteroni and his associates in~\cite{Quarteroni2021-3,Quarteroni2021-2,Quarteroni2021-1} and the references therein.

In this paper we consider a specific \emph{obstacle problem} for such a shell, in the sense that the shell is also subjected to a \emph{confinement condition}, expressing that any \emph{admissible displacement vector field} $v^\varepsilon_i \bm{g}^{i, \varepsilon}$ must be such that all the points of the corresponding deformed configuration remain in a \emph{half-space} of the form
\[
\mathbb{H} := \{ x \in \mathbb{E}^3; \, \boldsymbol{Ox} \cdot \bm{q} \ge 0\},
\]
where $\bm{q}$ is a \emph{nonzero vector} given once and for all. In other words, any admissible displacement field must satisfy
\[
\left( \bm{\Theta} (x^\varepsilon) + v^\varepsilon_i (x^\varepsilon) \bm{g}^{i, \varepsilon} (x^\varepsilon ) \right) \cdot \bm{q} \ge 0,
\]
for all $x^\varepsilon \in \overline{\Omega^\varepsilon}$, or possibly only for almost all (a.a. in what follows) $x^\varepsilon \in \Omega^\varepsilon$ when the covariant components $v^\varepsilon_i$ are required to belong to the Sobolev space $H^1(\Omega^\varepsilon)$ as in Theorem \ref{t:2} below.

We will of course assume that the reference configuration satisfies the confinement condition, i.e., that
\[
\bm{\Theta} (\overline{\Omega^\varepsilon}) \subset \mathbb{H}.
\]

It is to be emphasized that the above confinement condition \emph{considerably departs} from the usual \emph{Signorini condition} favored by most authors, who usually require that only the points of the undeformed and deformed ``lower face'' $\omega \times \{-\varepsilon\}$ of the reference configuration satisfy the confinement condition (see, e.g., \cite{LegMia2008}, \cite{LegMia2018}, \cite{Rodri2018}). Clearly, the confinement condition considered in the present paper, which is inspired by the formulation proposed by Br\'ezis \& Stampacchia~\cite{BrezisStampacchia1968}, is more physically realistic, since a Signorini condition imposed only on the lower face of the reference configuration does not prevent -- at least ``mathematically'' -- other points of the deformed reference configuration to ``cross'' the plane  $\{ x \in \mathbb{E}^3; \; \bm{Ox} \cdot \bm{q} = 0\}$ and then to end up on the ``other side'' of this plane (cf., e.g., Chapter~63 in~\cite{Zeidler1986}). It is evident that the vector $\bm{q}$ is thus orthogonal to the plane associated with the half-space where the linearly elastic shell is required to remain confined.

Such a confinement condition renders the asymptotic analysis considerably more difficult, however, as the constraint now bears on a vector field, the displacement vector field of the reference configuration, instead of on only a single component of this field.

The mathematical modeling of such an \emph{obstacle problem for a linearly elastic shell} is then clear; since, \emph{apart from} the confinement condition, the rest, i.e., the \emph{function space} and the expression of the quadratic \emph{energy} $J^\varepsilon$, is classical (see, e.g.~\cite{Ciarlet2000}). More specifically, let
\[
A^{ijk\ell, \varepsilon} := \lambda g^{ij, \varepsilon} g^{k\ell, \varepsilon} + \mu \left( g^{ik, \varepsilon} g^{j\ell, \varepsilon} + g^{i\ell, \varepsilon} g^{jk, \varepsilon} \right) =
  A^{jik\ell, \varepsilon} =  A^{k\ell ij, \varepsilon},
\]
denote the contravariant components of the \emph{elasticity tensor} of the elastic material constituting the shell. Then the unknown of the problem, which is the vector field $\bm{u}^\varepsilon = (u^\varepsilon_i)$ where the functions $u^\varepsilon_i : \overline{\Omega^\varepsilon} \to \mathbb{R}$ are the three covariant components of the unknown ``three-dimensional'' displacement vector field $u^\varepsilon_i \bm{g}^{i, \varepsilon}$ of the reference configuration of the shell, should minimize the \emph{energy} $J^\varepsilon : \bm{H}^1(\Omega^\varepsilon) \to \mathbb{R}$ defined by
\[
J^\varepsilon (\bm{v}^\varepsilon) := \frac12 \int_{\Omega^\varepsilon} A^{ijk\ell, \varepsilon} e^\varepsilon_{k\| \ell}  (\bm{v}^\varepsilon)e^\varepsilon_{i\|j} (\bm{v}^\varepsilon) \sqrt{g^\varepsilon} \dd x^\varepsilon - \int_{\Omega^\varepsilon} f^{i, \varepsilon} v^\varepsilon_i \sqrt{g^\varepsilon} \dd x^\varepsilon,
\]
for each $\bm{v}^\varepsilon = (v^\varepsilon_i) \in \bm{H}^1(\Omega^\varepsilon)$
over the \emph{set of admissible displacements} defined by:
$$
 \bm{U} (\Omega^\varepsilon  ) := \{ \bm{v}^\varepsilon = (v^\varepsilon_i) \in \bm{H}^1 (\Omega^\varepsilon) ; \; \bm{v}^\varepsilon = \textbf{0} \text{ on } \Gamma^\varepsilon_0, ( \bm{\Theta} (x^\varepsilon) + v^\varepsilon_i (x^\varepsilon) \bm{g}^{i, \varepsilon} (x^\varepsilon)) \cdot \bm{q} \ge 0 \textup{ for a.a. } x^\varepsilon \in \Omega^\varepsilon  \}.
$$

The solution to this \emph{minimization problem} exists and is unique, and it can be also characterized as the unique solution of the following problem:

\begin{customprob}{$\mathcal{P}(\Omega^\varepsilon)$}\label{problem0}
	Find $\bm{u}^\varepsilon =(u_i^\varepsilon)\in \bm{U}(\Omega^\varepsilon)$ that satisfies the following variational inequalities:
	$$
	\int_{\Omega^\varepsilon}
	A^{ijk\ell, \varepsilon} e^\varepsilon_{k\| \ell}  (\bm{u}^\varepsilon)
	\left( e^\varepsilon_{i\| j}  (\bm{v}^\varepsilon) -  e^\varepsilon_{i\| j}  (\bm{u}^\varepsilon)  \right) \sqrt{g^\varepsilon} \dd x^\varepsilon \ge \int_{\Omega^\varepsilon} f^{i , \varepsilon} (v^\varepsilon_i - u^\varepsilon_i)\sqrt{g^\varepsilon} \dd x^\varepsilon,
	$$
	for all $\bm{v}^\varepsilon = (v^\varepsilon_i) \in \bm{U}(\Omega^\varepsilon)$.
	\bqed	
\end{customprob}

The following result can be thus straightforwardly proved.

\begin{theorem} 
\label{t:2}
 The quadratic minimization problem: Find a vector field $\bm{u}^\varepsilon \in \bm{U} (\Omega^\varepsilon)$ such that
$$
J^\varepsilon (\bm{u}^\varepsilon) = \inf_{\bm{v}^\varepsilon \in \bm{U} (\Omega^\varepsilon)} J^\varepsilon (\bm{v}^\varepsilon),
$$
has one and only one solution. Besides, $\bm{u}^\varepsilon$ is also the unique solution of Problem~\ref{problem0}.
\end{theorem}
\begin{proof}
Define the space
\[
\bm{V}(\Omega^\varepsilon) := \{\bm{v}^\varepsilon = (v^\varepsilon_i) \in \bm{H}^1(\Omega^\varepsilon) ; \; \bm{v}^\varepsilon = \textbf{0} \text{ on } \Gamma^\varepsilon_0\}.
\]
Then, thanks to the uniform positive-definiteness of the elasticity tensor $(A^{ijk\ell, \varepsilon})$ \cite{Ciarlet1988}, and to the boundary condition of place satisfied on $\Gamma^\varepsilon_0 = \gamma_0 \times \left[- \varepsilon, \varepsilon\right] $ (recall that $\lambda \ge 0, \, \mu > 0$, and that $\gamma_0$ is a non-empty relatively open subset of $\gamma=\partial \omega$), it can be shown (see Theorems~3.8-3 and~3.9-1 of~\cite{Ciarlet2005}) that the continuous and symmetric bilinear form
\[
(\bm{v}^\varepsilon, \bm{w}^\varepsilon) \in \bm{H}^1(\Omega^\varepsilon)\times \bm{H}^1(\Omega^\varepsilon) \mapsto \int_{\Omega^\varepsilon} A^{ijk\ell, \varepsilon} e^\varepsilon_{k\| \ell} (\bm{v}^\varepsilon) e^\varepsilon_{i\| j}(\bm{w}^\varepsilon) \sqrt{g^\varepsilon} \dd x^\varepsilon,
\]
is $\bm{V} (\Omega^\varepsilon)$-elliptic; besides, the linear form
\[
\bm{v}^\varepsilon \in \bm{H}^1(\Omega^\varepsilon) \mapsto \int_{\Omega^\varepsilon} f^{i, \varepsilon} v^\varepsilon_i \sqrt{g^\varepsilon} \dd x^\varepsilon,
\]
is clearly continuous. Finally, the set $\bm{U}(\Omega^\varepsilon)$ is nonempty (by assumption), closed in $\bm{H}^1(\Omega^\varepsilon)$ (any convergent sequence in $\bm{V}(\Omega^\varepsilon)$ contains a subsequence that pointwise converges almost everywhere to its limit), and convex (as is immediately verified).

The existence and uniqueness of the solution to the  minimization problem and its characterization by means of variational inequalities is then classical (see, e.g., \cite{PGCLNFAA}, \cite{DuvLio76} or~\cite{Glow84}). 
\end{proof}

Since $\bm{\theta} (\overline{\omega}) \subset \bm{\Theta} (\overline{\Omega^\varepsilon})$, it evidently follows that $\bm{\theta} (y) \cdot \bm{q} \ge 0$ for all $y \in \overline{\omega}$. But in fact, a stronger property holds (cf. Lemma~2.1 of~\cite{CiaMarPie2018}).

\begin{lemma} \label{lem:1}
  Let $\omega$ be a domain in $\mathbb{R}^2$, let $\bm{\theta} \in \mathcal{C}^1(\overline{\omega}; \mathbb{E}^3)$ be an immersion, let $\bm{q} \in \mathbb{E}^3$ be a non-zero vector, and let $\varepsilon > 0$. Then the inclusion
  \[
\bm{\Theta} (\overline{\Omega^\varepsilon} ) \subset \mathbb{H} = \{ x \in \mathbb{E}^3; \; \boldsymbol{Ox} \cdot \bm{q} \ge 0\}
\]
implies that
\[
 \inf_{y \in \overline{\omega}} (\bm{\theta} (y) \cdot \bm{q}) > 0.
\]
\qed
\end{lemma}

We now consider the ``penalized'' version of Problem~\ref{problem0}. One such penalization transforms the set of variational inequalities in Problem~\ref{problem0} into a set of nonlinear variational equations, where the nonlinearity is defined in terms of the measure of the ``violation'' of the constraint.
Let $\kappa>0$ denote a \emph{penalty parameter}. The ``penalized'' variational formulation corresponding to Problem~\ref{problem0} takes the following form:

\begin{customprob}{$\mathcal{P}_{\kappa}(\Omega^\varepsilon)$}\label{problem1}
	Find $\bm{u}_{\kappa}^\varepsilon =(u_{i,\kappa}^\varepsilon) \in \bm{V}(\Omega^\varepsilon)$ that satisfies the following variational equations:
	\begin{align*}
	\int_{\Omega^\varepsilon} A^{ijk\ell, \varepsilon} e^\varepsilon_{k\| \ell}  (\bm{u}_\kappa^\varepsilon) e^\varepsilon_{i\| j}(\bm{v}^\varepsilon) \sqrt{g^\varepsilon} \dd x^\varepsilon
	-\dfrac{\varepsilon}{\kappa}\int_{\Omega^\varepsilon}\dfrac{\left\{[\bm{\Theta}+u_{j,\kappa}^\varepsilon \bm{g}^{j,\varepsilon}]\cdot \bm{q}\right\}^{-}}{\sqrt{\sum_{\ell=1}^{3}|\bm{g}^{\ell,\varepsilon}\cdot\bm{q}|^2}} (v_i^\varepsilon \bm{g}^{i,\varepsilon}\cdot \bm{q}) \sqrt{g^\varepsilon} \dd x^\varepsilon
	= \int_{\Omega^\varepsilon} f^{i , \varepsilon} v^\varepsilon_i\sqrt{g^\varepsilon} \dd x^\varepsilon,
	\end{align*}
	for all $\bm{v}^\varepsilon = (v^\varepsilon_i) \in \bm{V}(\Omega^\varepsilon)$.
	\bqed	
\end{customprob}

Note that the penalty term corresponds to the operator $\bm{\beta}^\varepsilon:\bm{L}^2(\Omega^\varepsilon) \to \bm{L}^2(\Omega^\varepsilon)$ defined by:
\begin{equation}
\label{betaeps}
\bm{\beta}^\varepsilon(\bm{v}):=\left(-\{(\bm{\Theta}+v_j \bm{g}^{j,\varepsilon})\cdot\bm{q}\}^{-}
\dfrac{\bm{g}^{i,\varepsilon}\cdot\bm{q}}{\sqrt{\sum_{\ell=1}^{3}|\bm{g}^{\ell,\varepsilon}\cdot\bm{q}|^2}}\right)_{i=1}^3,\quad\textup{ for all }\bm{v}=(v_i) \in \bm{L}^2(\Omega^\varepsilon).
\end{equation}

Note that the denominator appearing in the definition of $\bm{\beta}^\varepsilon$ is always positive, since the three vector fields $\bm{g}^{i,\varepsilon}$ are linearly independent at each $x^\varepsilon \in \overline{\Omega^\varepsilon}$.

In order to show the existence and uniqueness of solution for Problem~\ref{problem1}, we have to show that the operator $\bm{\beta}^\varepsilon$ is monotone. The following lemma, whose proof can be found, for instance, in~\cite{PieTem2023} serves for this purpose.

\begin{lemma}
	\label{lem:3}
	Let $\Omega \subset \mathbb{R}^n$, with $n\ge 1$ an integer, be an open set. The operator $-\{\cdot\}:L^2(\Omega) \to L^2(\Omega)$ defined by
	$$
	f \in L^2(\Omega) \mapsto -\{f\}^{-}:=\min\{f,0\} \in L^2(\Omega),
	$$
	is monotone, bounded and Lipschitz continuous with Lipschitz constant equal to $1$.
	\qed
\end{lemma}

The existence and uniqueness of the solution for Problem~\ref{problem1} is classical too (cf., e.g., Theorem~3.15 in~\cite{KikuchiOden1988}, or~\cite{Lions1969}).

\begin{theorem}
\label{ex-problem1}
Problem~\ref{problem1} has a unique solution.
\end{theorem}
\begin{proof}
Observe that the linear form
$$
\bm{v}\in \bm{H}^1(\Omega^\varepsilon) \mapsto \int_{\Omega^\varepsilon} f^{i, \varepsilon} v_i \sqrt{g^\varepsilon} \dd x^\varepsilon,
$$
is continuous (cf. Theorem~\ref{t:2}). The mapping $A:\bm{V}(\Omega^\varepsilon) \to \bm{V}'(\Omega^\varepsilon)$ defined by
$$
\langle A\bm{u},\bm{v}\rangle_{\bm{V}'(\Omega^\varepsilon), \bm{V}(\Omega^\varepsilon)}:=
\int_{\Omega^\varepsilon} A^{ijk\ell, \varepsilon} e^\varepsilon_{k\| \ell}  (\bm{u}) e^\varepsilon_{i\| j}(\bm{v}) \sqrt{g^\varepsilon} \dd x^\varepsilon
-\dfrac{\varepsilon}{\kappa}\int_{\Omega^\varepsilon}\dfrac{\left\{[\bm{\Theta}+u_{j,\kappa}^\varepsilon \bm{g}^{j,\varepsilon}]\cdot \bm{q}\right\}^{-}}{\sqrt{\sum_{\ell=1}^{3}|\bm{g}^{\ell,\varepsilon}\cdot\bm{q}|^2}} (v_i^\varepsilon \bm{g}^{i,\varepsilon}\cdot \bm{q}) \sqrt{g^\varepsilon} \dd x^\varepsilon,
$$
is continuous and, thanks to an \emph{ad hoc} Korn's inequality in curvilinear coordinates (cf., e.g., Theorem~1.7-4 of~\cite{Ciarlet2000}) and Lemma~\ref{lem:3}, strictly monotone and coercive. Therefore, the conclusion follows by the Minty-Browder theorem (cf., e.g., Theorem~9.14-1 of ~\cite{PGCLNFAA}).
\end{proof}

By means of a different classical approach based on energy estimates~\cite{Lions1969, Scholz1986} (see also part~(iv) of Theorem~9.14-1 of ~\cite{PGCLNFAA}), one can prove Theorem~\ref{ex-problem1} by establishing that for each $\varepsilon>0$
\begin{equation}
\label{weak}
\bm{u}^\varepsilon_\kappa \rightharpoonup \bm{u}^\varepsilon,\quad\textup{ in }\bm{V}(\Omega^\varepsilon),
\end{equation}
and that, thanks to the monotonicity of the penalty term established in Lemma~\ref{lem:3}, the weak limit satisfies Problem~\ref{problem0}. As a result of the convergence~\eqref{weak} and the continuity of the linear operator $e^\varepsilon_{i\| j}:\bm{H}^1(\Omega^\varepsilon) \to L^2(\Omega^\varepsilon)$, we have that
$$
e^\varepsilon_{i\| j}(\bm{u}^\varepsilon_\kappa) \rightharpoonup e^\varepsilon_{i\| j}(\bm{u}^\varepsilon), \quad \textup{ in } L^2(\Omega^\varepsilon).
$$

Actually, a stronger conclusion holds: Given $\varepsilon>0$, the following \emph{strong} convergence holds
\begin{equation*}
\label{strong}
\bm{u}^\varepsilon_\kappa \to \bm{u}^\varepsilon, \quad\textup{ in }\bm{V}(\Omega^\varepsilon).
\end{equation*}

To see this, observe that the uniform positive-definiteness of the elasticity tensor $(A^{ijk\ell,\varepsilon})$ (cf., e.g., \cite{Ciarlet1988}), Korn's inequality (viz. Theorem~1.7-4 of~\cite{Ciarlet2000}), the monotonicity of the penalty term (Lemma~\ref{lem:3}), the continuity of the operators $e_{i\|j}^\varepsilon$, and the weak convergence~\eqref{weak} give
\begin{align*}
\|\bm{u}^\varepsilon_\kappa - \bm{u}^\varepsilon\|_{\bm{H}^1(\Omega^\varepsilon)}^2 &\le
\dfrac{c_0 c_e}{\sqrt{g_0}}\int_{\Omega^\varepsilon} A^{ijk\ell, \varepsilon} e^\varepsilon_{k\| \ell}(\bm{u}^\varepsilon_\kappa - \bm{u}^\varepsilon) e^\varepsilon_{i\| j}(\bm{u}^\varepsilon_\kappa - \bm{u}^\varepsilon) \sqrt{g^\varepsilon} \dd x^\varepsilon\\
&=\dfrac{\varepsilon c_0 c_e}{\sqrt{g_0} \kappa}\int_{\Omega^\varepsilon} \dfrac{\left\{[\bm{\Theta}+u_{j,\kappa}^\varepsilon \bm{g}^{j,\varepsilon}]\cdot \bm{q}\right\}^{-}}{\sqrt{\sum_{\ell=1}^{3}|\bm{g}^{\ell,\varepsilon}\cdot\bm{q}|^2}} ((u^\varepsilon_{i,\kappa}-u^\varepsilon_i) \bm{g}^{i,\varepsilon}\cdot \bm{q}) \sqrt{g^\varepsilon} \dd x^\varepsilon\\
&\quad+\dfrac{c_0 c_e}{\sqrt{g_0}} \int_{\Omega^\varepsilon} f^{i, \varepsilon} (u^\varepsilon_{i,\kappa}-u^\varepsilon_i) \sqrt{g^\varepsilon} \dd x^\varepsilon\\
&\quad-\dfrac{c_0 c_e}{\sqrt{g_0}}\int_{\Omega^\varepsilon} A^{ijk\ell, \varepsilon} e^\varepsilon_{k\| \ell}(\bm{u}^\varepsilon) e^\varepsilon_{i\| j}(\bm{u}^\varepsilon_\kappa - \bm{u}^\varepsilon) \sqrt{g^\varepsilon} \dd x^\varepsilon\\
&\le \dfrac{c_0 c_e}{\sqrt{g_0}} \int_{\Omega^\varepsilon} f^{i, \varepsilon} (u^\varepsilon_{i,\kappa}-u^\varepsilon_i) \sqrt{g^\varepsilon} \dd x^\varepsilon\\
&\quad-\dfrac{c_0 c_e}{\sqrt{g_0}}\int_{\Omega^\varepsilon} A^{ijk\ell, \varepsilon} e^\varepsilon_{k\| \ell}(\bm{u}^\varepsilon) e^\varepsilon_{i\| j}(\bm{u}^\varepsilon_\kappa - \bm{u}^\varepsilon) \sqrt{g^\varepsilon} \dd x^\varepsilon \to 0,
\end{align*}
as $\kappa \to 0$.
In particular, we have that for each $\varepsilon>0$ and for each $\delta>0$, we can find a number $\kappa_0=\kappa_0(\delta,\varepsilon)>0$ such that, for each $0<\kappa<\kappa_0$, it results
\begin{equation}\label{lim:0}
\|\bm{u}^\varepsilon-\bm{u}^\varepsilon_\kappa\|_{\bm{H}^1(\Omega^\varepsilon)} <\dfrac{\delta}{2},
\end{equation}
for all $\varepsilon>0$, where $\bm{u}^\varepsilon$ and $\bm{u}^\varepsilon_\kappa$ respectively denote the solutions of Problem~\ref{problem0} and Problem~\ref{problem1}.

We observe that if $\bm{\beta}^\varepsilon(\bm{u}^\varepsilon_\kappa)$ enjoys higher regularity (viz., e.g., \cite{Scholz1986} and~\cite{MeiPie2022}) then the threshold number $\kappa_0$ can be made independent of $\varepsilon$.

\begin{theorem}
	\label{th:beta-6}
	Let $\kappa>0$ be given.
	Let $\bm{u}^\varepsilon$ be the solution of Problem~\ref{problem0} and let $\bm{u}^\varepsilon_\kappa$ be the solution of Problem~\ref{problem1}.
	
	Assume that there exists a constant $C_1>0$ independent of $\kappa$ and $\varepsilon$ such that $\|\bm{\beta}^\varepsilon(\bm{u}_\kappa^\varepsilon)\|_{\bm{V}(\Omega^\varepsilon)} \le C_1\sqrt{\kappa}$, where the nonlinear operator $\bm{\beta}^\varepsilon$ has been defined in~\eqref{betaeps}.

Then, there exists a constant $C>0$ independent of $\varepsilon$ and $\kappa$ such that:
	\begin{equation*}
		\|\bm{u}^\varepsilon-\bm{u}^\varepsilon_\kappa\|_{\bm{V}(\Omega^\varepsilon)}\le C \sqrt{\kappa}.
	\end{equation*}
\end{theorem}
\begin{proof}
	For each $\bm{v}\in \bm{L}^2(\Omega^\varepsilon)$, define the operator $\tilde{\bm{\beta}}^\varepsilon:\bm{L}^2(\Omega^\varepsilon) \to \bm{L}^2(\Omega^\varepsilon)$ by:
	\begin{equation*}
		\tilde{\bm{\beta}}^\varepsilon(\bm{v}):=\left(-\{(\bm{\Theta}+v_j\bm{g}^{j,\varepsilon})\cdot\bm{q}\}^{-}\left(\dfrac{\bm{g}^{i,\varepsilon}\cdot\bm{q}}{\sum_{\ell=1}^3|\bm{g}^{\ell,\varepsilon}\cdot\bm{q}|^2}\right)\right)_{i=1}^3.
	\end{equation*}
	
	Define $P(\bm{u}^\varepsilon_\kappa):=\bm{u}^\varepsilon_\kappa-\tilde{\bm{\beta}}^\varepsilon(\bm{u}^\varepsilon_\kappa)$, and observe that $P(\bm{u}^\varepsilon_\kappa) \in\bm{U}(\Omega^\varepsilon)$. Indeed, a direct computation gives
	\begin{equation*}
		\begin{aligned}
			&\left(\bm{\Theta}+\left[u^\varepsilon_{\kappa,i}-\dfrac{-\{(\bm{\Theta}+u^\varepsilon_{\kappa,j}\bm{g}^{j,\varepsilon})\cdot\bm{q}\}^{-}(\bm{g}^{i,\varepsilon}\cdot\bm{q})}{\sum_{\ell=1}^3|\bm{g}^{\ell,\varepsilon}\cdot\bm{q}|^2}\right]\bm{g}^{i,\varepsilon}\right)\cdot\bm{q}
			=((\bm{\Theta}+u^\varepsilon_{\kappa,i}\bm{g}^{i,\varepsilon})\cdot\bm{q})+\{(\bm{\Theta}+u^\varepsilon_{\kappa,i}\bm{g}^{i,\varepsilon})\cdot\bm{q}\}^{-}\\
			&=\{(\bm{\Theta}+u^\varepsilon_{\kappa,i}\bm{g}^{i,\varepsilon})\cdot\bm{q}\}^{+}\ge0.
		\end{aligned}
	\end{equation*}

	Let us estimate
	\begin{equation*}
		\label{est-1}
		\|\bm{u}^\varepsilon_\kappa-\bm{u}^\varepsilon\|_{\bm{V}(\Omega^\varepsilon)} \le \|\tilde{\bm{\beta}}^\varepsilon(\bm{u}^\varepsilon_\kappa)\|_{\bm{V}(\Omega^\varepsilon)}
		+\|\bm{u}^\varepsilon_\kappa-\tilde{\bm{\beta}}^\varepsilon(\bm{u}^\varepsilon_\kappa)-\bm{u}^\varepsilon\|_{\bm{V}(\Omega^\varepsilon)} \le C_1 \sqrt{\kappa}+\|\bm{u}^\varepsilon_\kappa-\tilde{\bm{\beta}}^\varepsilon(\bm{u}^\varepsilon_\kappa)-\bm{u}^\varepsilon\|_{\bm{V}(\Omega^\varepsilon)},
	\end{equation*}
	where the latter inequality holds thanks to the assumed estimate for the penalty term. Since $P(\bm{u}^\varepsilon_\kappa)\in \bm{U}(\Omega^\varepsilon)$, an application of the uniform positive definiteness of the fourth order three-dimensional elasticity tensor $(A^{ijk\ell,\varepsilon})$ (Theorem~1.8-1 of~\cite{Ciarlet2000}), Korn's inequality (Theorem~1.7-4 in~\cite{Ciarlet2000}) gives
	\begin{align*}
		&\dfrac{\sqrt{g_0}}{c_e c_0}\|P(\bm{u}^\varepsilon_\kappa)-\bm{u}^\varepsilon\|_{\bm{V}(\Omega^\varepsilon)}^2
		\le \int_{\Omega^\varepsilon}A^{ijk\ell,\varepsilon} e_{k\|\ell}^\varepsilon(P(\bm{u}^\varepsilon_\kappa)-\bm{u}^\varepsilon) e_{i\|j}^\varepsilon(P(\bm{u}^\varepsilon_\kappa)-\bm{u}^\varepsilon) \sqrt{g^\varepsilon} \dd x^\varepsilon\\
		&\le-\int_{\Omega^\varepsilon} \bm{f}^{\varepsilon} \cdot (P(\bm{u}^\varepsilon_\kappa)-\bm{u}^\varepsilon) \sqrt{g^\varepsilon} \dd x^\varepsilon
		+\int_{\Omega^\varepsilon} A^{ijk\ell,\varepsilon} e_{k\|\ell}^\varepsilon(P(\bm{u}^\varepsilon_\kappa)) e_{i\|j}^\varepsilon(P(\bm{u}^\varepsilon_\kappa)-\bm{u}^\varepsilon) \sqrt{g^\varepsilon} \dd x^\varepsilon\\
		&=-\int_{\Omega^\varepsilon} \bm{f}^{\varepsilon} \cdot (P(\bm{u}^\varepsilon_\kappa)-\bm{u}^\varepsilon) \sqrt{g^\varepsilon} \dd x^\varepsilon
		-\dfrac{\varepsilon}{\kappa} \int_{\Omega^\varepsilon} \bm{\beta}^\varepsilon(\bm{u}^\varepsilon_\kappa) \cdot (P(\bm{u}^\varepsilon_\kappa)-\bm{u}^\varepsilon) \sqrt{g^\varepsilon} \dd x^\varepsilon
		+\int_{\Omega^\varepsilon} \bm{f}^{\varepsilon} \cdot (P(\bm{u}^\varepsilon_\kappa)-\bm{u}^\varepsilon) \sqrt{g^\varepsilon} \dd x^\varepsilon\\
		&\quad-\int_{\Omega^\varepsilon}A^{ijk\ell,\varepsilon} e_{k\|\ell}^\varepsilon(\tilde{\bm{\beta}}^\varepsilon(\bm{u}^\varepsilon_\kappa)) e_{i\|j}^\varepsilon(P(\bm{u}^\varepsilon_\kappa)-\bm{u}^\varepsilon) \sqrt{g^\varepsilon} \dd x^\varepsilon\\
		&=\dfrac{\varepsilon}{\kappa}\int_{\Omega^\varepsilon}\left(-\{(\bm{\Theta}+u^\varepsilon_{\kappa,j}\bm{g}^{j,\varepsilon})\cdot\bm{q}\}^{-}\right) \left( \dfrac{u^\varepsilon_i\bm{g}^{i,\varepsilon}\cdot\bm{q}}{\sqrt{\sum_{\ell=1}^3|\bm{g}^{\ell,\varepsilon}\cdot\bm{q}|^2}}\right) \sqrt{g^\varepsilon} \dd x^\varepsilon\\
		&\quad-\dfrac{\varepsilon}{\kappa}\int_{\Omega^\varepsilon}\left(-\{(\bm{\Theta}+u^\varepsilon_{\kappa,j}\bm{g}^{j,\varepsilon})\cdot\bm{q}\}^{-}\right) \left( \dfrac{u^\varepsilon_{\kappa,i}\bm{g}^{i,\varepsilon}\cdot\bm{q}}{\sqrt{\sum_{\ell=1}^3|\bm{g}^{\ell,\varepsilon}\cdot\bm{q}|^2}}\right) \sqrt{g^\varepsilon} \dd x^\varepsilon\\
		&\quad+\dfrac{\varepsilon}{\kappa}\int_{\Omega^\varepsilon} \left(-\{(\bm{\Theta}+u^\varepsilon_{\kappa,j}\bm{g}^{j,\varepsilon})\cdot\bm{q}\}^{-} \dfrac{\bm{g}^{i,\varepsilon}\cdot\bm{q}}{\sqrt{\sum_{\ell=1}^3|\bm{g}^{\ell,\varepsilon}\cdot\bm{q}|^2}}\right)_{i=1}^3 \cdot \left(-\{(\bm{\Theta}+u^\varepsilon_{\kappa,j}\bm{g}^{j,\varepsilon})\cdot\bm{q}\}^{-} \dfrac{\bm{g}^{i,\varepsilon}\cdot\bm{q}}{\sum_{\ell=1}^3|\bm{g}^{\ell,\varepsilon}\cdot\bm{q}|^2}\right)_{i=1}^3 \sqrt{g^\varepsilon} \dd x^\varepsilon\\
		&\quad-\int_{\Omega^\varepsilon}A^{ijk\ell,\varepsilon} e_{k\|\ell}^\varepsilon(\tilde{\bm{\beta}}^\varepsilon(\bm{u}^\varepsilon_\kappa)) e_{i\|j}^\varepsilon(P(\bm{u}^\varepsilon_\kappa)-\bm{u}^\varepsilon) \sqrt{g^\varepsilon} \dd x^\varepsilon\\
		&\le-\dfrac{\varepsilon}{\kappa}\int_{\Omega^\varepsilon}\left(-\{(\bm{\Theta}^\varepsilon+u^\varepsilon_{\kappa,j}\bm{g}^{j,\varepsilon})\cdot\bm{q}\}^{-}\right) \left( \dfrac{\bm{\Theta}\cdot\bm{q}}{\sqrt{\sum_{\ell=1}^3|\bm{g}^{\ell,\varepsilon}\cdot\bm{q}|^2}}\right) \sqrt{g^\varepsilon} \dd x^\varepsilon\\
		&\quad+\dfrac{\varepsilon}{\kappa}\int_{\Omega^\varepsilon}\left(-\{(\bm{\Theta}+u^\varepsilon_{\kappa,j}\bm{g}^{j,\varepsilon})\cdot\bm{q}\}^{-}\right) \left( \dfrac{\bm{\Theta}\cdot\bm{q}}{\sqrt{\sum_{\ell=1}^3|\bm{g}^{\ell,\varepsilon}\cdot\bm{q}|^2}}\right) \sqrt{g^\varepsilon} \dd x^\varepsilon\\
		&\quad-\dfrac{\varepsilon}{\kappa}\int_{\Omega^\varepsilon}\dfrac{|-\{(\bm{\Theta}+u^\varepsilon_{\kappa,j}\bm{g}^{j,\varepsilon})\cdot\bm{q}\}^{-}|^2}{\sqrt{\sum_{\ell=1}^3|\bm{g}^{\ell,\varepsilon}\cdot\bm{q}|^2}} \sqrt{g^\varepsilon}\dd x^\varepsilon+\dfrac{\varepsilon}{\kappa}\int_{\Omega^\varepsilon}\dfrac{|-\{(\bm{\Theta}+u^\varepsilon_{\kappa,j}\bm{g}^{j,\varepsilon})\cdot\bm{q}\}^{-}|^2}{\sqrt{\sum_{\ell=1}^3|\bm{g}^{\ell,\varepsilon}\cdot\bm{q}|^2}} \sqrt{g^\varepsilon} \dd x^\varepsilon\\
		&\quad-\int_{\Omega^\varepsilon}A^{ijk\ell,\varepsilon} e_{k\|\ell}^\varepsilon(\tilde{\bm{\beta}}^\varepsilon(\bm{u}^\varepsilon_\kappa)) e_{i\|j}^\varepsilon(P(\bm{u}^\varepsilon_\kappa)-\bm{u}^\varepsilon) \sqrt{g^\varepsilon} \dd x^\varepsilon\\
		&=-\int_{\Omega^\varepsilon}A^{ijk\ell,\varepsilon} e_{k\|\ell}^\varepsilon(\tilde{\bm{\beta}}^\varepsilon(\bm{u}^\varepsilon_\kappa)) e_{i\|j}^\varepsilon(P(\bm{u}^\varepsilon_\kappa)-\bm{u}^\varepsilon) \sqrt{g^\varepsilon} \dd x^\varepsilon\le M \|\tilde{\bm{\beta}}^\varepsilon(\bm{u}^\varepsilon_\kappa)\|_{\bm{V}(\Omega^\varepsilon)} \|P(\bm{u}^\varepsilon_\kappa)-\bm{u}^\varepsilon\|_{\bm{V}(\Omega^\varepsilon)} \sqrt{g_1}\\
		&\le M C_1 \sqrt{g_1} \sqrt{\kappa} \|P(\bm{u}^\varepsilon_\kappa)-\bm{u}^\varepsilon\|_{\bm{V}(\Omega^\varepsilon)},
	\end{align*}
	where we recall that $g_0$ and $g_1$ are positive constants independent of $\varepsilon$ (viz., e.g., Theorem~3.1-1 of~\cite{Ciarlet2000}) and that the bounding constant $M$ for the fourth order three-dimensional elasticity tensor is independent of $\varepsilon$ (viz., e.g., Theorem~3.3-2 of~\cite{Ciarlet2000}).
	
	In conclusion, we have that
	\begin{equation*}
		\|P(\bm{u}^\varepsilon_\kappa)-\bm{u}^\varepsilon\|_{\bm{V}(\Omega^\varepsilon)} \le M C_1 c_0 c_e\dfrac{\sqrt{g_1}}{\sqrt{g}_0} \sqrt{\kappa},
	\end{equation*}
	so that
	\begin{equation*}
		\|\bm{u}^\varepsilon_\kappa-\bm{u}^\varepsilon\|_{\bm{V}(\Omega^\varepsilon)} \le C \sqrt{\kappa},
	\end{equation*}
	where the constant $C$ is defined by
	$$
	C:=C_1\left(1+M c_0 c_e \sqrt{\dfrac{g_1}{g_0}}\right).
	$$
	
	This completes the proof.
\end{proof}

The property that there exists a constant $C_1>0$ independent of $\kappa$ and $\varepsilon$ such that $\|\bm{\beta}^\varepsilon(\bm{u}_\kappa^\varepsilon)\|_{\bm{V}(\Omega^\varepsilon)} \le C_1\sqrt{\kappa}$, where $\bm{\beta}^\varepsilon$ has been defined in~\eqref{betaeps} could be established upon proving, in the same spirit of~\cite{Scholz1986} (see also~\cite{MeiPie2022}), that the solution $\bm{u}^\varepsilon$ of Problem~\ref{problem0} is of class $\bm{H}^2(\Omega^\varepsilon) \cap \bm{U}(\Omega^\varepsilon)$ and that the solution $\bm{u}^\varepsilon_\kappa$ is also of class $\bm{H}^2(\Omega^\varepsilon) \cap \bm{V}(\Omega^\varepsilon)$. However, the latter augmentations of regularity are not easy at all to prove, as the boundary conditions for the corresponding problems are only enforced on a portion of the boundary, thus preventing us from applying the argument of Agmon, Douglis \& Nirenberg~\cite{AgmDouNir1959,AgmDouNir1964}.

\section{The scaled three-dimensional problem for a family of flexural shells} \label{Sec:3}

In section~\ref{Sec:2}, we considered an obstacle problem for ``general'' linearly elastic shells. From now on, we will restrict ourselves  to a specific class of shells, according to the following definition (proposed in \cite{CiaLodsMia1996}; see also \cite{Ciarlet2000}).

Consider a linearly elastic shell, subjected to the various assumptions set forth in section~\ref{Sec:2}. Such a shell is said to be a \emph{linearly elastic flexural shell} if the following two additional assumptions are satisfied: \emph{first}, $\emptyset \neq \gamma_0 \subset \gamma$, i.e., the homogeneous boundary condition of place is imposed over a \emph{nonzero area portion of the entire lateral face} $\gamma_0 \times \left[ - \varepsilon , \varepsilon \right]$ of the shell, and \emph{second}, the space
\begin{align*}
\bm{V}_F(\omega):=\{\bm{\eta}=(\eta_i) \in H^1(\omega)\times H^1(\omega)\times H^2(\omega); e_{i\|j}(\bm{\eta})=0 \textup{ in }\omega
\textup{ and }\eta_i=\partial_{\nu}\eta_3=0 \textup{ on }\gamma_0\}
\end{align*}
contains nonzero functions, i.e., $\bm{V}_F(\omega)\neq\{\bf0\}$.

In this paper, we consider the \emph{obstacle problem} as defined in section~\ref{Sec:2} \emph{for a family of linearly elastic flexural shells}, all sharing the \emph{same middle surface} and whose thickness $2 \varepsilon > 0$ is considered as a ``small'' parameter approaching zero. Our objective then consists in performing an \emph{asymptotic analysis as} $\varepsilon \to 0$, so as to seek whether we can identify a \emph{limit two-dimensional problem}. To this end, we shall resort to a (by now standard) methodology first proposed by Ciarlet, Lods and Miara (cf.\ Theorem~5.1 of~\cite{CiaLodsMia1996} and Theorem~6.2-1 of~\cite{Ciarlet2000}): To begin with, we ``scale'' each problem $\mathcal{P} (\Omega^\varepsilon), \, \varepsilon > 0$, over a \emph{fixed domain} $\Omega$, using appropriate \emph{scalings on the unknowns} and \emph{assumptions on the data}. Note that these scalings and assumptions definitely \emph{depend} on the type of shells that are considered; for instance, those used for the \emph{linearly elastic elliptic membrane shells} considered elsewhere (cf.\ \cite{CiaMarPie2018b} and also~\cite{CiaMarPie2018}) are different.

More specifically, let
$$
\Omega := \omega \times \left] - 1, 1 \right[ ,
$$
let $x = (x_i)$ denote a generic point in the set $\overline{\Omega}$, and let $\partial_i := \partial/ \partial x_i$. With each point $x = (x_i) \in \overline{\Omega}$, we associate  the point $x^\varepsilon = (x^\varepsilon_i)$ defined by
$$
x^\varepsilon_\alpha := x_\alpha = y_\alpha \quad\text{ and }\quad x^\varepsilon_3 := \varepsilon x_3,
$$
so that $\partial^\varepsilon_\alpha = \partial_\alpha$ and $\partial^\varepsilon_3 = \displaystyle\frac{1}{\varepsilon} \partial_3$. To the unknown $\bm{u}^\varepsilon = (u^\varepsilon_i)$ and to the vector fields $\bm{v}^\varepsilon = (v^\varepsilon_i)$ appearing in the formulation of Problem~$\mathcal{P} (\Omega^\varepsilon)$ corresponding to a linearly elastic flexural shell, we then associate the \emph{scaled unknown} $\bm{u} (\varepsilon) = (u_i(\varepsilon))$ and the \emph{scaled vector fields} $\bm{v} = (v_i)$ by letting
$$
u_i (\varepsilon) (x) := u^\varepsilon_i (x^\varepsilon) \text{ and } v_i(x) := v^\varepsilon_i (x^\varepsilon),
$$
at each $x\in \overline{\Omega}$. Finally, we \emph{assume} that there exist functions $f^i \in L^2(\Omega)$ \emph{independent on} $\varepsilon$ such that the following \emph{assumptions on the data} hold:
\begin{equation*}
f^{i, \varepsilon} (x^\varepsilon) = \varepsilon^2 f^i(x) \text{ at each } x \in \Omega.
\end{equation*}

Note that the independence on $\varepsilon$ of the Lam\'{e} constants assumed in section~\ref{Sec:2} in the formulation of Problem~\ref{problem0} implicitly constituted another \emph{assumption on the data}.

In view of the proposed scaling, we define the ``scaled'' version of the geometrical entities introduced in section~\ref{Sec:1}:
\begin{align*}
	\bm{g}^i(\varepsilon ) (x) &:= \bm{g}^{i, \varepsilon} (x^\varepsilon) \textup{ at each } x\in \overline{\Omega},\\
	g(\varepsilon) (x) &:= g^\varepsilon (x^\varepsilon)
	\text{ and } A^{ijk\ell} (\varepsilon) (x) :=
	A^{ijk\ell, \varepsilon}(x^\varepsilon)
	\text{ at each } x\in \overline{\Omega}, \\
	e_{\alpha \| \beta} (\varepsilon; \bm{v}) &:= \frac12 (\partial_\beta v_\alpha + \partial_\alpha v_\beta ) - \Gamma^k_{\alpha \beta} (\varepsilon) v_k = e_{\beta \| \alpha} (\varepsilon; \bm{v}) , \\
	e_{\alpha \| 3} (\varepsilon; \bm{v}) = e_{3\| \alpha}(\varepsilon; \bm{v}) &:= \frac12 \left(\frac{1}{\varepsilon} \partial_3 v_\alpha + \partial_\alpha v_3 \right) - \Gamma^\sigma_{\alpha3} (\varepsilon) v_\sigma, \\
	e_{3\| 3} (\varepsilon ; \bm{v}) &:= \frac{1}{\varepsilon} \partial_3 v_3,
\end{align*}
where
\[
\Gamma^p_{ij} (\varepsilon) (x) := \Gamma^{p, \varepsilon}_{ij} (x^\varepsilon) \textup{ at each } x\in \overline{\Omega}.
\]

Define the space
$$
\bm{V}(\Omega):=\{\bm{v} = (v_i) \in \bm{H}^1(\Omega) ; \; \bm{v} = \textbf{0} \textup{ on } \gamma_0 \times \left[ -1, 1 \right]\},
$$
and define, for each $\varepsilon > 0$, the set
\begin{align*}
	\bm{U} (\varepsilon; \Omega) &:= \{\bm{v} = (v_i) \in \bm{V}(\Omega); \big(\bm{\theta} (y) + \varepsilon x_3 \bm{a}_3 (y) + v_i (x) \bm{g}^i(\varepsilon) (x)\big) \cdot \bm{q} \ge 0 \textup{ for a.a. } x = (y, x_3 ) \in \Omega\}.
\end{align*}

We are thus in a position to introduce the ``scaled'' version of Problem~\ref{problem0}, that will be denoted in what follows by $\mathcal{P}(\varepsilon; \Omega)$.

\begin{customprob}{$\mathcal{P} (\varepsilon; \Omega)$}
\label{problem0scaled}
Find $\bm{u}(\varepsilon)=(u_i(\varepsilon)) \in \bm{U}(\varepsilon; \Omega)$ that satisfies the following variational inequalities:
\begin{align*}
	\int_\Omega A^{ijk\ell}(\varepsilon) e_{k\| \ell} (\varepsilon; \bm{u}(\varepsilon)) \left(e_{i\| j} (\varepsilon; \bm{v}) - e_{i\|j} (\varepsilon; \bm{u}(\varepsilon))\right) \sqrt{g(\varepsilon)} \dd x
	\ge \varepsilon^2 \int_\Omega f^i (v_i - u_i(\varepsilon)) \sqrt{g (\varepsilon)} \dd x,
\end{align*}
for all $\bm{v}=(v_i) \in \bm{U} (\varepsilon; \Omega)$.
\bqed
\end{customprob}

\begin{theorem} \label{t:3}
The scaled unknown $\bm{u}(\varepsilon)$ is the unique solution of the variational Problem~\ref{problem0scaled}.
\end{theorem}
\begin{proof}
The variational Problem~\ref{problem0scaled} simply constitutes a re-writing of the variational Problem~\ref{problem0}, this time in terms of the scaled unknown $\bm{u}(\varepsilon)$, of the vector fields $\bm{v}$, and of the functions $f^i$, which are now all defined over the domain ${\Omega}$. Then the assertion follows from this observation.
\end{proof}

The functions $e_{i\|j} (\varepsilon; \bm{v})$ appearing in Problem~\ref{problem0scaled} are called the \emph{scaled linearized strains in curvilinear coordinates} associated with the scaled displacement vector field $v_i \bm{g}^i(\varepsilon)$.

For later purposes (like in Lemma \ref{lem:2} below), we also let
$$
\bm{g}_i (\varepsilon)(x) := \bm{g}^\varepsilon_i (x^\varepsilon) \textup{ at each } x \in \overline{\Omega}.
$$

Likewise, one can introduce the ``scaled'' version of Problem~\ref{problem1}, that will be denoted in what follows by $\mathcal{P}_{\kappa}(\varepsilon; \Omega)$.

\begin{customprob}{$\mathcal{P}_{\kappa}(\varepsilon; \Omega)$}
	\label{problem1scaled}
	Find $\bm{u}_\kappa(\varepsilon) =(u_{i,\kappa}(\varepsilon)) \in \bm{V}(\Omega)$ that satisfies the following variational equations:
	\begin{align*}
		&\int_\Omega A^{ijk\ell} (\varepsilon) e_{k\| \ell} (\varepsilon; \bm{u}_{\kappa}(\varepsilon)) e_{i\| j} (\varepsilon; \bm{v}) \sqrt{g(\varepsilon)} \dd x\\
		&\quad -\dfrac{\varepsilon}{\kappa}\int_{\Omega} \dfrac{\left\{[\bm{\theta} + \varepsilon x_3 \bm{a}_3 + u_{j,\kappa}(\varepsilon) \bm{g}^j(\varepsilon)]\cdot \bm{q}\right\}^{-}}{\sqrt{\sum_{\ell=1}^{3}|\bm{g}^\ell(\varepsilon) \cdot \bm{q}|^2}} (v_i \bm{g}^i(\varepsilon)\cdot\bm{q}) \sqrt{g(\varepsilon)} \dd x
		= \varepsilon^2 \int_\Omega f^i v_i \sqrt{g (\varepsilon)} \dd x,
	\end{align*}
	for all $\bm{v}=(v_i) \in \bm{V}(\Omega)$.
	\bqed
\end{customprob}

The following existence and uniqueness result can be thus easily proved in the same fashion as Theorem~\ref{ex-problem1}.

\begin{theorem} \label{t:3bis}
The scaled unknown $\bm{u}_{\kappa}(\varepsilon)$ is the unique solution of the variational Problem~\ref{problem1scaled}.
\end{theorem}
\begin{proof}
The variational Problem~\ref{problem1scaled} simply constitutes a re-writing of the variational Problem~\ref{problem1}, this time in terms of the scaled unknown $\bm{u}_\kappa(\varepsilon)$, of the vector fields $\bm{v}$, and of the functions $f^i$, which are now all defined over the domain $\Omega$. Then the assertion follows from this observation.
\end{proof}

By means of an analogous reasoning, a condition similar to~\eqref{lim:0} can be derived, i.e., for each $\delta>0$ we can find a number $\kappa_0=\kappa_0(\delta,\varepsilon)>0$ such that, for each $0<\kappa<\kappa_0$, it results
\begin{equation}\label{lim:1}
	\|\bm{u}(\varepsilon)-\bm{u}_\kappa(\varepsilon)\|_{\bm{H}^1(\Omega)} <\dfrac{\delta}{2},
\end{equation}
for each $\varepsilon>0$, where $\bm{u}(\varepsilon)$ and $\bm{u}_\kappa(\varepsilon)$ respectively denote the solutions of Problem~\ref{problem0scaled} and Problem~\ref{problem1scaled}.

In the same spirit as Theorem~\ref{th:beta-6} it can be shown that if the scaled solution $\bm{u}(\varepsilon)$ enjoys higher regularity then the threshold value $\kappa_0$ can be made independent of $\varepsilon$.

Without loss of generality (cf., e.g., \cite{Scholz1986}), given any $\varepsilon>0$, we restrict ourselves to considering penalty parameters with the following property:
\begin{equation}
\label{kappa}
0<\kappa= \sqrt{\varepsilon}.
\end{equation}

It is straightforward to observe that $\kappa \to 0$ as $\varepsilon \to 0$.
We observe that the variational Problem~\ref{problem0scaled} could have been equivalently written as a \emph{minimization problem}, thus mimicking that found in Theorem \ref{t:2}. It turns out, however, that its  formulation in Theorem \ref{t:3bis} as a set of \emph{penalized variational equations} is more convenient for the asymptotic analysis undertaken in section~\ref{Sec:4}.

It is immediately verified (cf., e.g., \cite{Ciarlet2000}) that \emph{other} assumptions on the data are possible that would give rise to the \emph{same} problem over the fixed domain $\Omega$. For instance, should the Lam\'{e} constants (now denoted) $\lambda^\varepsilon$ and $\mu^\varepsilon$ appearing in Problem~\ref{problem0} be of the form $\lambda^\varepsilon = \varepsilon^t \lambda$ and $\mu^\varepsilon = \varepsilon^t \mu$, where $\lambda \ge 0$ and $\mu$ are constants independent of $\varepsilon$ and $t$ is an arbitrary real number, the \emph{same} Problem~\ref{problem0scaled} arises if we assume that the components of the applied body force density are now of the form
$$
f^{i, \varepsilon} (x^\varepsilon) = \varepsilon^{2+t} f^i(x) \text{ at each } x\in \Omega,
$$
where the functions $f^i \in L^2(\Omega)$ are independent of $\varepsilon$.

The next lemma assembles various asymptotic properties as $\varepsilon \to 0$ of functions and vector fields appearing in the formulation of Problem~\ref{problem0scaled}; these properties will be repeatedly used in the proof of the convergence theorem (Theorem \ref{t:5}).

In the statement of the next preparatory lemma (cf., e.g., Theorems~3.3-1 and~3.3-2 of~\cite{Ciarlet2000}), the notation ``$O(\varepsilon)$'', or ``$O(\varepsilon^2)$'', stands for a remainder that is of order $\varepsilon$, or $\varepsilon^2$, with respect to the sup-norm over the set $\overline{\Omega}$, and any function, or vector-valued function, of the variable $y \in \overline{\omega}$, such as $a^{\alpha \beta}, b_{\alpha \beta}, \bm{a}^i$, etc. (all these are defined in section~\ref{Sec:1}) is identified with the function, or vector-valued function, of $x = (y, x_3) \in \overline{\Omega} = \overline{\omega} \times \left[-1, 1\right]$ that takes the same value at $x_3 = 0$ and is independent of $x_3 \in \left[-1, 1\right]$; for brevity, this extension from $\overline{\omega} $ to $\overline{\Omega}$ is designated with the same notation.

Recall that $\varepsilon > 0$ is implicitly assumed to be small enough so that $\bm{\Theta} : \overline{\Omega^\varepsilon} \to \mathbb{E}^3$ is an immersion.

\begin{lemma} \label{lem:2}
Let $\bm{\theta} \in \mathcal{C}^3(\overline{\omega};\mathbb{R}^3)$ be an immersion.
Let $\varepsilon_0$ be defined as in Theorem~3.1-1 of~\cite{Ciarlet2000}. The functions $A^{ijk\ell} (\varepsilon) = A^{jik\ell} (\varepsilon) = A^{k\ell ij} (\varepsilon)$ have the following properties:
\[
A^{ijk\ell} (\varepsilon) = A^{ijk\ell} ( 0 ) + O(\varepsilon) , \quad A^{\alpha \beta \sigma 3} (\varepsilon) = A^{\alpha 333} (\varepsilon ) = 0,
\]
for all $0<\varepsilon \le \varepsilon_0$, where
\begin{align*}
  A^{\alpha \beta \sigma \tau} (0) &= \lambda a^{\alpha \beta} a^{\sigma \tau} + \mu (a^{\alpha \sigma}a^{\beta \tau} + a^{\alpha \tau} a^{\beta \sigma}), \\
  A^{\alpha \beta 33} (0) &= \lambda a^{\alpha \beta} , \quad A^{\alpha 3 \sigma 3} (0) = \mu a^{\alpha \sigma}, \quad A^{3333}(0) = \lambda + 2 \mu,
\end{align*}
and there exists a constant $C_0 > 0$ such that
\[
\sum_{i, j} \left| t_{ij} \right|^2 \leq C_0 A^{ijk\ell} (\varepsilon) (x) t_{k\ell} t_{ij}
\]
for all $0<\varepsilon \le \varepsilon_0$, all $x \in \overline{\Omega}$, and all symmetric matrices $(t_{ij})$.

The functions $\Gamma^p_{ij} (\varepsilon)$ and $g(\varepsilon)$ have the following properties:
\begin{align*}
  \Gamma^\sigma_{\alpha \beta} (\varepsilon) &= \Gamma^\sigma_{\alpha \beta} - \varepsilon x_3 (\partial_\alpha b^\sigma_\beta + \Gamma^\sigma_{\alpha \tau} b^\tau_\beta - \Gamma^\tau_{\alpha \beta} b^\sigma_\tau) + O(\varepsilon^2), \\
  \Gamma^3_{\alpha \beta} (\varepsilon) &= b_{\alpha \beta} - \varepsilon x_3 b^\sigma_\alpha b_{\sigma \beta}, \quad \partial_3 \Gamma^p_{\alpha \beta} (\varepsilon) = O(\varepsilon), \\
  \Gamma^\sigma_{\alpha 3} (\varepsilon) &= - b^\sigma_\alpha - \varepsilon x_3 b^\tau_\alpha b^\sigma_\tau + O(\varepsilon^2), \quad \Gamma^3_{\alpha 3} (\varepsilon) = \Gamma^p_{33} (\varepsilon ) =0, \\
  g(\varepsilon) &= a + O(\varepsilon),
\end{align*}
for all $0<\varepsilon \le \varepsilon_0$ and all $x \in \overline{\Omega}$. In particular then, there exist constants $g_0$ and $g_1$ such that
\[
0 < g_0 \leq g(\varepsilon)(x) \leq g_1 \textup{ for all }0<\varepsilon \le \varepsilon_0 \textup{ and all }x \in \overline{\Omega}.
\]

The vector fields $\bm{g}_i (\varepsilon)$ and $\bm{g}^j(\varepsilon)$ have the following properties:
\begin{align*}
  \bm{g}_\alpha (\varepsilon) &= \bm{a}_\alpha - \varepsilon x_3 b^\sigma_\alpha \bm{a}_\sigma, \quad \bm{g}_3 (\varepsilon) = \bm{a}_3 , \\
\bm{g}^\alpha (\varepsilon) &= \bm{a}^\alpha + \varepsilon x_3 b^\alpha_\sigma \bm{a}^\sigma + O(\varepsilon^2), \quad \bm{g}^3(\varepsilon) = \bm{a}^3.
\end{align*}
\qed
\end{lemma}

We recall (cf., e.g., \cite{Ciarlet2000}), that the various relations and estimates in Lemma~\ref{lem:2} hold in fact for \emph{any} family of linearly elastic shells, i.e., irrespective of whether these shells are flexural ones or not.

When one considers a family of linearly elastic \emph{flexural shells} whose thickness $2 \varepsilon$ approaches zero, a specific \emph{Korn's inequality in curvilinear coordinates} (cf., e.g., Theorem~4.1 of~\cite{CiaLodsMia1996} or Theorem~5.3-1 of~\cite{Ciarlet2000}) holds over the \emph{fixed} domain $\Omega = \omega \times \left] -1,1 \right[$, according to the following theorem. That the constant $C_1$ that appears in this inequality is \emph{independent of} $\varepsilon > 0$ plays a key role in the asymptotic analysis of such a family (see part (i) of the proof of Theorem~\ref{t:5}).

\begin{theorem}\label{t:4}
Let $\bm{\theta} \in \mathcal{C}^3(\overline{\omega};\mathbb{R}^3)$ be an immersion.
Let there be given a family of linearly elastic flexural shells with the same middle surface $\bm{\theta} (\overline{\omega})$ and thickness $2 \varepsilon > 0$. Define the space
$$
\bm{V} (\Omega) := \{\bm{v} = (v_i) \in \bm{H}^1(\Omega) ; \; \bm{v} = \bm{0} \textup{ on }  \gamma_0  \times \left[-1,1\right]\}.
$$
Then there exist constants $\varepsilon_1 > 0$ and $C_1 > 0$ such that
$$
  \left\{ \sum_i \left\| v_i \right\|^2_{H^1(\Omega)}\right\}^{1/2} \le \dfrac{C_1}{\varepsilon} \left\{ \sum_{i,j} \left\| e_{i\|j} (\varepsilon; \bm{v}) \right\|^2_{L^2(\Omega)} \right\}^{1/2}
$$
for all $0 < \varepsilon \leq \varepsilon_1$ and all $\bm{v}=(v_i) \in \bm{V} (\Omega)$.
\qed
\end{theorem}

\section{Rigorous asymptotic analysis}
\label{Sec:4}

The ultimate goal of this paper is to show, in the same spirit as~\cite{CiaLodsMia1996} (see also Theorem~6.2-1 of~\cite{Ciarlet2000}), that \emph{the solutions $\bm{u}(\varepsilon)$ of the} (scaled) \emph{three-dimensional problems~\ref{problem0} converge - as $\varepsilon$ approaches zero - to the solution of a two-dimensional problem}, denoted by $\mathcal{P}_F(\omega)$ in what follows.

The proof proposed in~\cite{CiaLodsMia1996} (see also Theorem~6.2-1 of~\cite{Ciarlet2000}) resorts, however, to the usage of a specific vector field that was first introduced by Miara and Sanchez-Palencia~\cite{MiaSan1996}. This construction argument is, in general, not applicable to the context of variational inequalities, for which the test functions are chosen in a nonempty, closed and convex subset of a certain space.

In order to overcome this difficulty, we first prove that, under the assumption~\eqref{kappa}, \emph{the solutions $\bm{u}_{\kappa}(\varepsilon)$ of} Problem~\ref{problem1scaled} converge - as $\varepsilon$ approaches zero - to the solution of the variational problem $\mathcal{P}_F(\omega)$.

Define the set
\begin{align*}
	\bm{U}_F(\omega) &:= \{\bm{\eta} = (\eta_i) \in \bm{V}_F(\omega);\big(\bm{\theta} (y) + \eta_i (y) \bm{a}^i(y)\big) \cdot \bm{q} \ge 0 \textup{ for a.a. } y \in \omega \},
\end{align*}
where the space $\bm{V}_F(\omega)$ has been defined in section~\ref{Sec:3}. We observe that the set $\bm{U}_F(\omega)$ is nonempty, closed and convex in the space $\bm{V}_F(\omega)$. The vector fields $\bm{a}^i$ and the functions $\Gamma^\sigma_{\alpha \beta}, b_{\alpha \beta}, a^{\alpha \beta}, a$, and $\gamma_{\alpha \beta}(\bm{\eta})$, have been defined in section~\ref{Sec:1}.
We are thus in a position to define the two-dimensional problem~$\mathcal{P}_F(\omega)$ as follows:

\begin{customprob}{$\mathcal{P}_F(\omega)$}
\label{problemLim}
Find $\bm{\zeta}=(\zeta_i) \in \bm{U}_F(\omega)$ that satisfies the following variational inequalities:
$$
\dfrac{1}{3}\int_\omega a^{\alpha \beta \sigma \tau} \rho_{\sigma \tau}(\bm{\zeta}) \rho_{\alpha \beta} (\bm{\eta} - \bm{\zeta}) \sqrt{a} \dd y \ge \int_\omega p^i (\eta_i - \zeta_i) \sqrt{a} \dd y,
$$
for all $\bm{\eta} = (\eta_i) \in \bm{U}_F(\omega)$, where
$$
a^{\alpha \beta \sigma \tau} := \frac{4\lambda \mu}{\lambda + 2 \mu} a^{\alpha \beta} a^{\sigma \tau} + 2\mu \left(a^{\alpha \sigma} a^{\beta \tau} + a^{\alpha \tau} a^{\beta \sigma}\right) \textup{ and } p^i := \int^1_{-1} f^i \dd x_3.
$$
\bqed
\end{customprob}

In the same spirit as Theorem~\ref{t:2}, it can be show that Problem~\ref{problemLim} admits one and only one solution.

We are now ready to show that, under the assumption~\eqref{kappa}, \emph{the solutions $\bm{u}_{\kappa}(\varepsilon)$ of} Problem~\ref{problem1scaled} converge - as $\varepsilon$ approaches zero - to the solution of Problem~\ref{problemLim}.

Differently from the linearly elastic elliptic membrane case, we will see that problems arise when one has to define an appropriate test vector field for recovering the limit model, which is the main objective of Theorem~\ref{t:5} (cf. part (iv)). Since this test vector field will have to depend on the partial derivatives of the displacement, it is hard to find an expression for it that takes into account the geometrical constraint as well. For this reason, the penalty method seems to be the most convenient technique to attack this problem.

\begin{theorem} \label{t:5}
  Let $\omega$ be a domain in $\mathbb{R}^2$, let $\bm{\theta} \in \mathcal{C}^3(\overline{\omega}; \mathbb{E}^3)$ be the middle surface of a \emph{flexural shell}, let $\gamma_0$ be a non-zero length portion of the boundary $\gamma$ (cf. section~\ref{Sec:3}) and let $\bm{q} \in \mathbb{E}^3$ be a non-zero vector given once and for all. Let us consider the \emph{non-trivial} space (cf. section~\ref{Sec:3})
\begin{equation*}
\bm{V}_F(\omega):=\{\bm{\eta}=(\eta_i) \in H^1(\omega)\times H^1(\omega)\times H^2(\omega); \gamma_{\alpha \beta}(\bm{\eta})=0 \textup{ in }\omega \textup{ and }\eta_i=\partial_{\nu}\eta_3=0 \textup{ on }\gamma_0\},
\end{equation*}
and let us define the set
\begin{align*}
\bm{U}_F (\omega) &:= \{\bm{\eta} = (\eta_i) \in \bm{V}_F(\omega);\big(\bm{\theta} (y) + \eta_i (y) \bm{a}^i(y)\big) \cdot \bm{q} \ge 0 \textup{ for a.a. } y \in \omega \}.
\end{align*}
%and assume that the immersion $\bm{\theta}$ is such that
%\[
%d:= \inf_{y\in \overline{\omega}} (\bm{\theta} (y) \cdot \bm{q}) > 0.
%\]

Let there be given a family of linearly elastic flexural shells with the same middle surface $\bm{\theta} (\overline{\omega})$ and thickness $2 \varepsilon > 0$, and let $\bm{u}_\kappa(\varepsilon) \in \bm{V}(\Omega)$ denote for each $\varepsilon > 0$ the unique solution of Problem~\ref{problem1scaled}, where the penalty parameter $\kappa$ is assumed to be as in~\eqref{kappa}.

Then there exists $\bm{u} \in \bm{H}^1(\Omega)$ independent of the variable $x_3$ and satisfying
\begin{align*}
&{\bm{u}=\bf0} \textup{ on }  \Gamma_0=\gamma_0 \times \left[-1, 1\right],\\
&\bm{u}_\kappa(\varepsilon) \to \bm{u} \textup{ in } \bm{H}^1(\Omega) \textup{ as } \varepsilon \to 0.
\end{align*}

Define the average
\[
\overline{\bm{u}} = (\overline{u}_i) := \frac12 \int^1_{-1} \bm{u} \dd x_3.
\]
Then
\[
\overline{\bm{u}} = \bm{\zeta},
\]
where $\bm{\zeta}$ is the unique solution to the two-dimensional variational Problem~\ref{problemLim}.
\end{theorem}

\begin{proof}

Strong and weak convergences as $\varepsilon \to 0$ are respectively denoted by $\to$ and $\rightharpoonup$. For brevity, we let
\[
e_{i\| j} (\varepsilon) := e_{i\|j} (\varepsilon ; \bm{u}_\kappa(\varepsilon)).
\]

The \emph{outline} of the proof, which is broken into six parts numbered (i)--(vi), is to a large extent inspired by the proof of Theorem~6.2-1 of~\cite{Ciarlet2000} (itself adapted from Theorem 5.1 in Ciarlet, Lods and Miara~\cite{CiaLodsMia1996}), where no confinement condition was imposed. This is why some parts of the proof are reminiscent of those in~\cite{Ciarlet2000}; otherwise, considering the confinement condition requires extra care.

(i) \emph{There exists a subsequence, still denoted $(\bm{u}_\kappa(\varepsilon))_{\varepsilon > 0}$, and there exists $\bm{u} \in \bm{H}^1(\Omega)$ and there exist $e_{i\|j}^1 \in L^2(\Omega)$ satisfying}
$$
{\bm{u}=\bf0} \textup{ on }  \Gamma_0=\gamma_0 \times \left[-1, 1\right]
$$
\emph{and such that}
\begin{gather*}
\bm{u}_\kappa(\varepsilon) \rightharpoonup \bm{u} \textup{ in }\bm{H}^1(\Omega) \textup{ and thus } \bm{u}_\kappa(\varepsilon) \to \bm{u} \textup{ in }\bm{L}^2(\Omega), \\
(\bm{\theta}(y)+u_i(y,x_3) \bm{a}^i(y)) \cdot \bm{q} \ge 0 \textup{ for a.a. }x=(y,x_3) \in \Omega,\\
\dfrac{1}{\varepsilon}e_{i\|j} (\varepsilon) \rightharpoonup e_{i\|j}^1 \textup{ in } L^2(\Omega).
\end{gather*}

Letting $\bm{v} = \bm{u}_\kappa(\varepsilon)$ in the variational equations of Problem~\ref{problem1scaled}. Combining the uniform positive-definiteness of the tensor $(A^{ijk\ell} (\varepsilon))$, the Korn inequality  of Theorem~\ref{t:4}, the asymptotic behavior of the function $g(\varepsilon)$ (Lemma~\ref{lem:2}), and the fact that (see Lemmas~\ref{lem:1} and~\ref{lem:3})
\begin{align*}
&-\dfrac{\varepsilon}{\kappa}\int_{\Omega}\dfrac{\left\{[\bm{\theta} + \varepsilon x_3 \bm{a}_3+u_{j,\kappa}(\varepsilon)\bm{g}^j(\varepsilon)]\cdot \bm{q}\right\}^{-}}{\sqrt{\sum_{\ell=1}^{3}|\bm{g}^\ell(\varepsilon) \cdot \bm{q}|^2}}(u_{i,\kappa}(\varepsilon)\bm{g}^i(\varepsilon)\cdot \bm{q}) \sqrt{g(\varepsilon)} \dd x\\
&=-\dfrac{\varepsilon}{\kappa}\dfrac{\int_{\Omega}\left\{[\bm{\theta} + \varepsilon x_3 \bm{a}_3+u_{j,\kappa}(\varepsilon)\bm{g}^j(\varepsilon)]\cdot \bm{q}\right\}^{-}}{\sqrt{\sum_{\ell=1}^{3}|\bm{g}^\ell(\varepsilon) \cdot \bm{q}|^2}}[(\bm{\theta} + \varepsilon x_3 \bm{a}_3+u_{i,\kappa}(\varepsilon)\bm{g}^i(\varepsilon))\cdot \bm{q}] \sqrt{g(\varepsilon)} \dd x\\
&\quad+\dfrac{\varepsilon}{\kappa}\int_\Omega\dfrac{\left\{[\bm{\theta} + \varepsilon x_3 \bm{a}_3+u_{j,\kappa}(\varepsilon)\bm{g}^j(\varepsilon)]\cdot \bm{q}\right\}^{-}}{\sqrt{\sum_{\ell=1}^{3}|\bm{g}^\ell(\varepsilon) \cdot \bm{q}|^2}} [(\bm{\theta} + \varepsilon x_3 \bm{a}_3)\cdot \bm{q}] \sqrt{g(\varepsilon)} \dd x\ge0, \quad \textup{ for all } \kappa>0\textup{ and all }\varepsilon>0,
\end{align*}
we obtain for  $\varepsilon > 0$ sufficiently small:
\begin{align*}
  C^{-2}_1 \varepsilon^2 \sum_i  \left\| u_{i,\kappa} (\varepsilon) \right\|^2_{H^1(\Omega)} &\leq \sum_{i, j} \left\| e_{i\|j} (\varepsilon) \right\|^2_{L^2(\Omega)} \leq\dfrac{C_0}{\sqrt{g_0}} \int_\Omega A^{ijk\ell} (\varepsilon) e_{k\|\ell} (\varepsilon ) e_{i\|j} (\varepsilon) \sqrt{g(\varepsilon)} \dd x \\
  &\quad-\dfrac{C_0 \varepsilon}{\kappa\sqrt{g_0}}\int_{\Omega}\dfrac{\left\{[\bm{\theta} + \varepsilon x_3 \bm{a}_3+u_{j,\kappa}(\varepsilon)\bm{g}^j(\varepsilon)]\cdot \bm{q}\right\}^{-}}{\sqrt{\sum_{\ell=1}^{3}|\bm{g}^\ell(\varepsilon) \cdot \bm{q}|^2}}(u_{i,\kappa}(\varepsilon)\bm{g}^i(\varepsilon)\cdot \bm{q}) \sqrt{g(\varepsilon)} \dd x\\
  & = \varepsilon^2 \frac{C_0}{\sqrt{g_0}} \int_\Omega f^i u_{i,\kappa}(\varepsilon) \sqrt{g(\varepsilon)} \dd x \\
  &\le \varepsilon^2 C_0 \sqrt{\frac{g_1}{g_0}} \Big\{ \sum_i \| f^i\|^2_{L^2(\Omega)} \Big\}^{1/2} \Big\{ \sum_i \left\| u_{i,\kappa}(\varepsilon) \right\|^2_{L^2(\Omega)} \Big\}^{1/2}.
\end{align*}

This chain of inequalities first shows that the norms $\left(\left\| u_{i,\kappa}(\varepsilon)\right\|_{H^1(\Omega)}\right)_{\varepsilon>0}$ are bounded independently of $\varepsilon$, secondly, that the terms $\left(\varepsilon^{-1}\left\| e_{i\|j} (\varepsilon) \right\|_{L^2(\Omega)}\right)_{\varepsilon>0}$ are bounded independently of $\varepsilon$ and, finally, that the terms
\begin{equation}
\label{penalty}
\left(\dfrac{1}{\sqrt{\varepsilon}}\left\|\dfrac{\left\{[\bm{\theta} + \varepsilon x_3 \bm{a}_3+u_{i,\kappa}(\varepsilon)\bm{g}^i(\varepsilon)]\cdot \bm{q}\right\}^{-}}{\kappa \sqrt[4]{\sum_{\ell=1}^{3}|\bm{g}^\ell(\varepsilon) \cdot \bm{q}|^2}}\right\|_{L^2(\Omega)}^2\right)_{\varepsilon>0}
\end{equation}
are bounded independently of $\varepsilon$ as well (recall that, by assumption~\eqref{kappa}, we have that $0<\kappa=\sqrt{\varepsilon}$). 

Hence there exist a subsequence, a vector field $\bm{u} \in \bm{H}^1(\Omega)$, and functions $e_{i\|j}^1 \in L^2(\Omega)$ such that, if we let $\varepsilon \to 0$, the following convergence process occurs:
\begin{equation}
\label{process1}
\begin{aligned}
\bm{u}_\kappa(\varepsilon) &\rightharpoonup \bm{u} \textup{ in } \bm{H}^1(\Omega),\\
\dfrac{1}{\varepsilon}e_{i\|j} (\varepsilon) &\rightharpoonup e_{i\|j}^1 \textup{ in } L^2(\Omega),\\
\dfrac{1}{\kappa}\dfrac{\left\{[\bm{\theta} + \varepsilon x_3 \bm{a}_3+u_{i,\kappa}(\varepsilon)\bm{g}^i(\varepsilon)]\cdot \bm{q}\right\}^{-}}{\sqrt[4]{\sum_{\ell=1}^{3}|\bm{g}^\ell(\varepsilon) \cdot \bm{q}|^2}} &\to 0 \textup{ in }L^2(\Omega),\\
\left\{[\bm{\theta} + \varepsilon x_3 \bm{a}_3+u_{i,\kappa}(\varepsilon)\bm{g}^i(\varepsilon)]\cdot \bm{q}\right\}^{-} &\to 0 \textup{ in }L^2(\Omega).
\end{aligned}
\end{equation}

The fact that $\bm{u}_\kappa(\varepsilon) \to \bm{u}   \textup{ in } \bm{L}^2(\Omega)$ is a consequence of the Rellich-Kondra\v{s}ov Theorem (viz., e.g., Theorem~6.6-3 of~\cite{PGCLNFAA}). Note that, by the definition of limit, the latter means that for any given $\delta>0$ there exists a number $\varepsilon_2=\varepsilon_2(\delta)>0$ such that for all $0<\varepsilon<\varepsilon_2$ it results:
\begin{equation}
\label{L2def}
\|\bm{u}_\kappa(\varepsilon) - \bm{u}\|_{\bm{L}^2(\Omega)} <\delta.
\end{equation}

Note that the third convergence in~\eqref{process1} means that for any given $\delta>0$ there exists a number $\check{\varepsilon}_2=\check{\varepsilon}_2(\delta)>0$ such that for all $0<\varepsilon<\check{\varepsilon}_2$ it results:
\begin{equation}
\label{penalty-L2-0}
\left\|\dfrac{\left\{[\bm{\theta} + \varepsilon x_3 \bm{a}_3+u_{i,\kappa}(\varepsilon)\bm{g}^i(\varepsilon)]\cdot \bm{q}\right\}^{-}}{\kappa \sqrt[4]{\sum_{\ell=1}^{3}|\bm{g}^\ell(\varepsilon) \cdot \bm{q}|^2}}\right\|_{L^2(\Omega)}<\delta.
\end{equation}

Moreover, recall that the term~\eqref{penalty} is bounded independently of $\varepsilon$, in the sense that there exists a constant $C=C(C_0,C_1,g_0,g_1,\omega,\gamma_0,\bm{\theta},\bm{f})>0$ such that:
\begin{equation}
\label{penalty-L2-02}
\left\|\dfrac{\left\{[\bm{\theta} + \varepsilon x_3 \bm{a}_3+u_{i,\kappa}(\varepsilon)\bm{g}^i(\varepsilon)]\cdot \bm{q}\right\}^{-}}{\kappa \sqrt[4]{\sum_{\ell=1}^{3}|\bm{g}^\ell(\varepsilon) \cdot \bm{q}|^2}}\right\|_{L^2(\Omega)}
\le C \varepsilon^{1/4}.
\end{equation}

Therefore, it suffices to take $\check{\varepsilon}_2=\delta^4/C$.

Note that the fourth convergence in~\eqref{process1} means that for any given $\delta>0$ there exists a number $\hat{\varepsilon}_2=\hat{\varepsilon}_2(\delta)>0$ such that for all $0<\varepsilon<\hat{\varepsilon}_2$ it results:
\begin{equation}
\label{penalty-L2}
\left\|\left\{[\bm{\theta} + \varepsilon x_3 \bm{a}_3+u_{i,\kappa}(\varepsilon)\bm{g}^i(\varepsilon)]\cdot \bm{q}\right\}^{-}\right\|_{L^2(\Omega)}<\delta.
\end{equation}

Moreover, recall that the term~\eqref{penalty} is bounded independently of $\varepsilon$, in the sense that there exists a constant $C=C(C_0,C_1,g_0,g_1,\omega,\gamma_0,\bm{\theta},\bm{f})>0$ such that:
\begin{equation}
\label{penalty-L2-2}
\left\|\left\{[\bm{\theta} + \varepsilon x_3 \bm{a}_3+u_{i,\kappa}(\varepsilon)\bm{g}^i(\varepsilon)]\cdot \bm{q}\right\}^{-}\right\|_{L^2(\Omega)}
\le C \kappa\varepsilon^{1/4}=C \varepsilon^{3/4},
\end{equation}
so that $\hat{\varepsilon}_2=\delta^{4/3}/C$ is a suitable threshold number for verifying the definition of limit.

Combining the convergence~\eqref{L2def} with the fourth convergence in~\eqref{process1}, the continuity of the negative part operator (Lemma~\ref{lem:3}), and the fact that, for a generic function,
$$
f^{-}=\dfrac{|f|-f}{2},
$$
gives:
$$
(\bm{\theta}(y)+u_i(y,x_3) \bm{a}^i(y))\cdot\bm{q} \ge 0,\quad\textup{ for a.a. }x=(y,x_3) \in \Omega.
$$

That $\bm{u}=\bf0$ on $\gamma_0 \times \left[-1,1\right]$ follows from the continuity of the trace operator $\mathop{\mathrm{tr}}: H^1(\Omega) \to L^2(\gamma \times \left[-1, 1\right])$. Indeed, for all $i$, we have that for all $v \in H^1(\Omega)$ such that $v=0$ on $\Gamma \setminus \Gamma_0$,
$$
0=\int_{\Gamma_0} u_{i,\kappa}(\varepsilon) v \dd\Gamma \to \int_{\Gamma_0} u_i v \dd\Gamma,
$$
so that a density result proved by Bernard~\cite{Ber2011} (see also Theorem~6.7-3 of~\cite{PGCLNFAA}) gives $u_i=0$ on $\Gamma_0$.

(ii) \emph{The weak limits $u_i \in H^1(\Omega)$ found in} (i) \emph{are independent of the variable $x_3 \in \left[ -1,1\right]$, in the sense that they satisfy
$$
\partial_3 u_i = 0 \textup{ in } L^2(\Omega).
$$
Besides, the average $\overline{\bm{u}}$ satisfies $\overline{\bm{u}} \in \bm{U}_F(\omega)$, namely,
\begin{align*}
\overline{\bm{u}}=(\overline{u}_i) \in H^1(\omega) \times H^1(\omega) \times &H^2(\omega) \textup{ and } \overline{u}_i=\partial_{\nu}\overline{u}_3=0 \textup{ on }\gamma_0,\\
\gamma_{\alpha \beta}(\overline{\bm{u}})&=0 \textup{ in }\omega,\\
\big(\bm{\theta} (y) + \overline{u}_i (y) \bm{a}^i(y)\big) &\cdot \bm{q} \ge 0 \textup{ for a.a. } y \in \omega.
\end{align*}
}

Apart from the latter property, the proof is identical to that of part (ii) of the proof of Theorem 6.2-1 in \cite{Ciarlet2000}. Let us thus prove that
$$
\big(\bm{\theta} (y) + \overline{u}_i (y) \bm{a}^i(y)\big) \cdot \bm{q} \ge 0 \textup{ for a.a. } y \in \omega.
$$

By part~(i), we have that
$$
(\bm{\theta}(y)+u_i(y,x_3)\bm{a}^i(y))\cdot \bm{q} \ge 0 \quad\textup{ for a.a. }x=(y,x_3) \in\Omega.
$$

Since $\bm{u}=(u_i)$ is independent of $x_3$, we have that an application of Theorem~4.2-1 (a) of~\cite{Ciarlet2000} and part~(i) gives
\begin{equation*}
\begin{aligned}
(\bm{\theta}(y)+\overline{u}_i(y)\bm{a}^i(y))\cdot \bm{q} &= \left(\bm{\theta}(y)+\dfrac{1}{2} \int_{-1}^{1} u_i(y,x_3) \dd x_3 \bm{a}^i(y)\right)\cdot \bm{q}\\
&=\dfrac{1}{2}\int_{-1}^{1}\left(\left(\bm{\theta}(y)+u_i(y,x_3)\bm{a}^i(y)\right)\cdot \bm{q}\right) \dd x_3\ge 0
,\quad\textup{ for a.a. }y \in\omega,
\end{aligned}
\end{equation*}
so that $\overline{\bm{u}} = (\overline{u}_i) \in \bm{U}_F(\omega)$.

(iii) \emph{The weak limits $e_{i\|j}^1 \in L^2(\Omega)$ and $\bm{u} \in \bm{H}^1(\Omega)$ found in} (i) \emph{satisfy}
\begin{align*}
&-\partial_3 e_{\alpha\|\beta}^1=\rho_{\alpha\beta}(\bm{u}) \textup{ in }L^2(\Omega),\\
&e_{\alpha \| 3}^1 = 0 \quad \textup{ and } \quad e_{3\|3}^1 = - \frac{\lambda}{\lambda + 2 \mu} a^{\alpha \beta} e_{\alpha \| \beta}^1 \textup{ in } \Omega.
\end{align*}
The equality $-\partial_3 e_{\alpha\|\beta}^1=\rho_{\alpha\beta}(\bm{u})$ in $ L^2(\Omega)$ follows from Theorem~5.2-2 of~\cite{Ciarlet2000}.
Let $\bm{v}=(v_i) \in \bm{V}(\Omega)$ be arbitrarily chosen. It is known (cf., e.g., part~(iii) of Theorem~6.2-1 of~\cite{Ciarlet2000}) that
\begin{align*}
\varepsilon e_{\alpha\|\beta}(\varepsilon;\bm{v}) &\to 0 \textup{ in } L^2(\Omega),\\
\varepsilon e_{\alpha\|3}(\varepsilon;\bm{v}) &\to \dfrac{1}{2} \partial_3 v_\alpha \textup{ in } L^2(\Omega),\\
\varepsilon e_{3\|3}(\varepsilon;\bm{v}) &=\partial_3 v_3 \textup{ for all }\varepsilon>0.
\end{align*}

These relations, combined with the boundedness of the terms $\left(\varepsilon^{-1}\left\| e_{i\|j} (\varepsilon) \right\|_{L^2(\Omega)}\right)_{\varepsilon>0}$ independently of $\varepsilon > 0$ (part (i)), the asymptotic behavior of the functions $A^{ijk\ell} (\varepsilon)$, the third convergence in~\eqref{process1}, and $\sqrt{g(\varepsilon)}$ as $\varepsilon \to 0$ (Lemma~\ref{lem:2}), give
\begin{align*}
&\int_\Omega  \left( A^{\alpha \beta \sigma \tau}(\varepsilon) \dfrac{e_{\sigma\|\tau}(\varepsilon)}{\varepsilon} 
+ A^{\alpha \beta 33}(\varepsilon) \dfrac{e_{3\|3}(\varepsilon)}{\varepsilon}\right) \varepsilon e_{\alpha\|\beta}(\varepsilon;\bm{v}) \sqrt{g(\varepsilon)} \dd x\to 0 \textup{ as } \varepsilon  \to   0, \\
& \int_\Omega \left(4 A^{\alpha 3 \sigma 3}(\varepsilon)\dfrac{e_{\sigma\|3}(\varepsilon)}{\varepsilon}\right) \varepsilon e_{\alpha\|3}(\varepsilon;\bm{v}) \sqrt{g(\varepsilon)} \dd x \to \int_\Omega 2 \mu a^{\alpha \sigma} e_{\sigma\|3}^1 \partial_3 v_\alpha \sqrt{a} \dd x \textup{ as } \varepsilon \to  0, \\
& \int_\Omega \left(A^{33\sigma \tau}(\varepsilon) \dfrac{e_{\sigma \| \tau}(\varepsilon)}{\varepsilon} + A^{3333}(\varepsilon)\dfrac{e_{3\|3}(\varepsilon)}{\varepsilon}\right) \varepsilon e_{3\|3}(\varepsilon;\bm{v}) \sqrt{g(\varepsilon)} \dd x\\
&\quad\to \int_\Omega \left( \lambda a^{\sigma \tau} e_{\sigma \| \tau}^1 + (\lambda + 2 \mu) e_{3\| 3}^1 \right) \partial_3 v_3 \sqrt{a} \dd x \textup{ as } \varepsilon  \to  0, \\
&-\dfrac{\varepsilon}{\kappa}\int_{\Omega} \dfrac{\left\{[\bm{\theta} + \varepsilon x_3 \bm{a}_3+u_{i,\kappa}(\varepsilon) \bm{g}^i(\varepsilon)]\cdot \bm{q}\right\}^{-}}{\sqrt{\sum_{\ell=1}^{3}|\bm{g}^\ell(\varepsilon) \cdot \bm{q}|^2}} (v_i \bm{g}^i(\varepsilon)\cdot\bm{q}) \sqrt{g(\varepsilon)} \dd x \to 0 \textup{ as } \varepsilon \to 0 \,(\textup{cf. } \eqref{kappa}\textup{--}\eqref{process1}),\\
&\varepsilon^2 \int_\Omega  f^i v_i\sqrt{g(\varepsilon)} \dd x \to 0 \textup{ as } \varepsilon \to 0.
\end{align*}

Consequently,
$$
\int_\Omega \left((2\mu a^{\alpha \sigma} e_{\sigma \| 3}^1) \partial_3 v_\alpha + \big( \lambda a^{\sigma \tau} e_{\sigma \| \tau}^1 + (\lambda + 2 \mu ) e_{3\| 3}^1 \big) \partial_3 v_3\right) \sqrt{a} \dd x = 0.
$$
Since this inequality holds for any vector field $\bm{v}=(v_i) \in \bm{V}(\Omega)$, it follows by Theorem~3.4-1 of~\cite{Ciarlet2000} that
$$
e_{\sigma \| 3}^1 = 0 \quad \textup{ and } \quad \lambda a^{\sigma \tau} e_{\sigma \| \tau}^1 + (\lambda + 2 \mu) e_{3\| 3}^1 = 0 \textup{ in } L^2(\Omega).
$$
In particular, the latter gives:
$$
e_{3\|3}^1 = - \frac{\lambda}{\lambda + 2 \mu} a^{\alpha \beta} e_{\alpha \| \beta}^1 \textup{ in } L^2(\Omega).
$$

(iv) \emph{The weak limit $\overline{\bm{u}}=(\overline{u}_i)$ is the unique solution of Problem~\ref{problemLim}.}
	
Given any $\bm{\eta} \in \bm{U}_F(\omega)$, define the vector field $\bm{w}(\varepsilon;\bm{\eta})=(w_i(\varepsilon;\bm{\eta}))$ by:
\begin{equation}
\label{w}
\begin{aligned}
w_\alpha(\varepsilon;\bm{\eta})&:=\eta_\alpha-\varepsilon x_3 (\partial_\alpha \eta_3 +b_\alpha^\sigma \eta_\sigma),\\
w_3(\varepsilon;\bm{\eta})&:=\eta_3.
\end{aligned}
\end{equation}
	
We have that $\bm{w}(\varepsilon;\bm{\eta}) \in \bm{V}(\Omega)$ and $e_{3\|3}(\varepsilon;\bm{w}(\varepsilon;\bm{\eta}))=0$ for all $\varepsilon>0$.
Following the same steps as in part (iv) of~\cite{Ciarlet2000}, observe that:
\begin{equation}
\label{weta}
\begin{aligned}
\bm{w}(\varepsilon;\bm{\eta}) \to& \bm{\eta} \textup{ in }\bm{H}^1(\Omega),\\
\dfrac{1}{\varepsilon} e_{\alpha\| \beta}(\varepsilon;\bm{w}(\varepsilon;\bm{\eta})) \to& \{-x_3 \rho_{\alpha \beta}(\bm{\eta})\} \textup{ in } L^2(\Omega),\\
\textup{The sequence }\left(\dfrac{1}{\varepsilon} e_{\alpha\| 3}(\varepsilon;\bm{w}(\varepsilon;\bm{\eta}))\right)_{\varepsilon>0} &\textup{ converges in }L^2(\Omega).
\end{aligned}
\end{equation}

	Fix $\varepsilon>0$ and define the number
	$$
	\Lambda(\varepsilon):=\int_\Omega A^{ijk \ell}(\varepsilon)\left\{\dfrac{1}{\varepsilon}e_{k\| \ell}(\varepsilon) -e_{k\|\ell}^1\right\} \left\{\dfrac{1}{\varepsilon}e_{i\| j}(\varepsilon) -e_{i\|j}^1\right\} \sqrt{g(\varepsilon)} \dd x,
	$$
	so that, combining the uniform positive-definiteness of the three-dimensional elasticity tensor $A^{ijk\ell}(\varepsilon)$ and the asymptotic behavior of the function $g(\varepsilon)$ (Lemma~\ref{lem:2}), we obtain:
	$$
	0\le \sum_{i, j} \left\|\dfrac{1}{\varepsilon}e_{i\|j}(\varepsilon)-e_{i\|j}^1\right\|_{L^2(\Omega)}^2\le\dfrac{C_0}{\sqrt{g_0}} \Lambda(\varepsilon).
	$$
	
	An application of the latter to the variational equations of Problem~\ref{problem1scaled} under the specialization \mbox{$\bm{v}=(\bm{w}(\varepsilon;\bm{\eta})-\bm{u}_\kappa(\varepsilon))$}, with $\bm{\eta} \in \bm{U}_F(\omega)$ gives:
	\begin{equation}
	\label{eqLe}
	\begin{aligned}
	0 \le \Lambda(\varepsilon)&=\int_\Omega A^{ijk\ell}(\varepsilon) \left\{\dfrac{1}{\varepsilon}e_{k\|\ell}(\varepsilon)\right\} \left\{\dfrac{1}{\varepsilon}e_{i\|j}(\varepsilon)\right\} \sqrt{g(\varepsilon)} \dd x\\
	&\quad -2 \int_\Omega A^{ijk\ell}(\varepsilon) \left\{\dfrac{1}{\varepsilon} e_{k\|\ell}(\varepsilon)\right\} e_{i\|j}^1 \sqrt{g(\varepsilon)} \dd x\\
	&\quad +\int_\Omega A^{ijk\ell}(\varepsilon) e_{k\|\ell}^1 e_{i\|j}^1 \sqrt{g(\varepsilon)} \dd x\\
	&= \int_\Omega A^{ijk\ell}(\varepsilon)\left\{\dfrac{1}{\varepsilon}e_{k\|\ell}(\varepsilon)\right\} \left\{\dfrac{1}{\varepsilon}e_{i\|j}(\varepsilon;\bm{w}(\varepsilon;\bm{\eta}))\right\} \sqrt{g(\varepsilon)} \dd x\\
	&\quad-\dfrac{1}{\varepsilon\kappa} \int_\Omega \dfrac{\left\{[\bm{\theta} + \varepsilon x_3 \bm{a}_3+u_{j,\kappa}(\varepsilon)\bm{g}^j(\varepsilon)]\cdot \bm{q}\right\}^{-}}{\sqrt{\sum_{\ell=1}^{3}|\bm{g}^\ell(\varepsilon) \cdot \bm{q}|^2}} \left((w_i(\varepsilon;\bm{\eta})-u_{i,\kappa}(\varepsilon))\bm{g}^i(\varepsilon)\cdot\bm{q}\right) \sqrt{g(\varepsilon)} \dd x\\
	&\quad -\int_\Omega f^i (w_i(\varepsilon;\bm{\eta})-u_{i,\kappa}(\varepsilon)) \sqrt{g(\varepsilon)} \dd x\\
	&\quad -2 \int_\Omega A^{ijk\ell}(\varepsilon) \left\{\dfrac{1}{\varepsilon} e_{k\|\ell}(\varepsilon)\right\} e_{i\|j}^1 \sqrt{g(\varepsilon)} \dd x+\int_\Omega A^{ijk\ell}(\varepsilon) e_{k\|\ell}^1 e_{i\|j}^1 \sqrt{g(\varepsilon)} \dd x\\
	&= \int_\Omega A^{\alpha\beta\sigma\tau}(\varepsilon)\left\{\dfrac{1}{\varepsilon}e_{\sigma\|\tau}(\varepsilon)\right\} \left\{\dfrac{1}{\varepsilon}e_{\alpha\|\beta}(\varepsilon;\bm{w}(\varepsilon;\bm{\eta}))\right\} \sqrt{g(\varepsilon)} \dd x\\
	&\quad+\int_\Omega A^{\alpha\beta33}(\varepsilon)\left\{\dfrac{1}{\varepsilon}e_{3\|3}(\varepsilon)\right\} \left\{\dfrac{1}{\varepsilon}e_{\alpha\|\beta}(\varepsilon;\bm{w}(\varepsilon;\bm{\eta}))\right\} \sqrt{g(\varepsilon)} \dd x\\
	&\quad+4\int_\Omega A^{\alpha 3 \sigma 3}(\varepsilon)\left\{\dfrac{1}{\varepsilon}e_{\sigma\|3}(\varepsilon)\right\} \left\{\dfrac{1}{\varepsilon}e_{\alpha\|3}(\varepsilon;\bm{w}(\varepsilon;\bm{\eta}))\right\} \sqrt{g(\varepsilon)} \dd x\\
	&\quad+\int_\Omega A^{33\sigma\tau}(\varepsilon)\left\{\dfrac{1}{\varepsilon}e_{\sigma\|\tau}(\varepsilon)\right\} \left\{\dfrac{1}{\varepsilon}e_{3\|3}(\varepsilon;\bm{w}(\varepsilon;\bm{\eta}))\right\} \sqrt{g(\varepsilon)} \dd x\\
	&\quad+\int_\Omega A^{3333}(\varepsilon)\left\{\dfrac{1}{\varepsilon}e_{3\|3}(\varepsilon)\right\} \left\{\dfrac{1}{\varepsilon}e_{3\|3}(\varepsilon;\bm{w}(\varepsilon;\bm{\eta}))\right\} \sqrt{g(\varepsilon)} \dd x\\
	&\quad-\dfrac{1}{\varepsilon\kappa} \int_\Omega \dfrac{\left\{[\bm{\theta} + \varepsilon x_3 \bm{a}_3+u_{j,\kappa}(\varepsilon)\bm{g}^j(\varepsilon)]\cdot \bm{q}\right\}^{-}}{\sqrt{\sum_{\ell=1}^{3}|\bm{g}^\ell(\varepsilon) \cdot \bm{q}|^2}} \left((w_i(\varepsilon;\bm{\eta})-u_{i,\kappa}(\varepsilon))\bm{g}^i(\varepsilon)\cdot\bm{q}\right) \sqrt{g(\varepsilon)} \dd x\\
	&\quad -\int_\Omega f^i (w_i(\varepsilon;\bm{\eta})-u_{i,\kappa}(\varepsilon)) \sqrt{g(\varepsilon)} \dd x\\
	&\quad -2 \int_\Omega A^{ijk\ell}(\varepsilon) \left\{\dfrac{1}{\varepsilon} e_{k\|\ell}(\varepsilon)\right\} e_{i\|j}^1 \sqrt{g(\varepsilon)} \dd x+\int_\Omega A^{ijk\ell}(\varepsilon) e_{k\|\ell}^1 e_{i\|j}^1 \sqrt{g(\varepsilon)} \dd x.
	\end{aligned}
	\end{equation}

By virtue of the definition of the vector field $\bm{w}(\varepsilon;\bm{\eta})=(w_i(\varepsilon;\bm{\eta}))$ introduced in~\eqref{w} and the asymptotic behaviour of the contravariant basis vectors $\bm{g}^i(\varepsilon)$ established in Lemma~\ref{lem:2}, we obtain that the term
\begin{equation}
	\label{eqLe-2}
	\begin{aligned}
		&-\dfrac{1}{\varepsilon}\int_{\Omega} \dfrac{\left\{[\bm{\theta} + \varepsilon x_3 \bm{a}_3+u_{j,\kappa}(\varepsilon)\bm{g}^j(\varepsilon)]\cdot \bm{q}\right\}^{-}}{\kappa \sqrt{\sum_{\ell=1}^{3}|\bm{g}^\ell(\varepsilon) \cdot \bm{q}|^2}} [(w_i(\varepsilon;\bm{\eta})-u_{i,\kappa}(\varepsilon))\bm{g}^i(\varepsilon) \cdot \bm{q}] \sqrt{g(\varepsilon)} \dd x\\
		&=-\dfrac{1}{\varepsilon}\int_{\Omega} \dfrac{\left\{[\bm{\theta} + \varepsilon x_3 \bm{a}_3+u_{j,\kappa}(\varepsilon)\bm{g}^j(\varepsilon)]\cdot \bm{q}\right\}^{-}}{\kappa\sqrt{\sum_{\ell=1}^{3}|\bm{g}^\ell(\varepsilon) \cdot \bm{q}|^2}}[(\bm{\theta} + \varepsilon x_3 \bm{a}_3+w_i(\varepsilon;\bm{\eta})\bm{g}^i(\varepsilon))\cdot \bm{q}] \sqrt{g(\varepsilon)} \dd x\\
		&\quad +\dfrac{1}{\varepsilon}\int_{\Omega} \dfrac{\left\{[\bm{\theta} + \varepsilon x_3 \bm{a}_3+u_{j,\kappa}(\varepsilon)\bm{g}^j(\varepsilon)]\cdot \bm{q}\right\}^{-}}{\kappa\sqrt{\sum_{\ell=1}^{3}|\bm{g}^\ell(\varepsilon) \cdot \bm{q}|^2}}\left([\bm{\theta} + \varepsilon x_3 \bm{a}_3+u_{i,\kappa}(\varepsilon)\bm{g}^i(\varepsilon)]\cdot \bm{q}\right) \sqrt{g(\varepsilon)} \dd x\\
		&=-\dfrac{1}{\varepsilon}\int_{\Omega} \dfrac{\left\{[\bm{\theta} + \varepsilon x_3 \bm{a}_3+u_{j,\kappa}(\varepsilon)\bm{g}^j(\varepsilon)]\cdot \bm{q}\right\}^{-}}{\kappa\sqrt{\sum_{\ell=1}^{3}|\bm{g}^\ell(\varepsilon) \cdot \bm{q}|^2}}[(\bm{\theta} + \eta_i \bm{a}^i) \cdot \bm{q}] \sqrt{g(\varepsilon)} \dd x\\
		&\quad-\dfrac{1}{\varepsilon}\int_{\Omega} \dfrac{\left\{[\bm{\theta} + \varepsilon x_3 \bm{a}_3+u_{j,\kappa}(\varepsilon)\bm{g}^j(\varepsilon)]\cdot \bm{q}\right\}^{-}}{\kappa\sqrt{\sum_{\ell=1}^{3}|\bm{g}^\ell(\varepsilon) \cdot \bm{q}|^2}}[\varepsilon x_3 (\eta_\alpha b_\sigma^\alpha \bm{a}^\sigma) \cdot \bm{q}] \sqrt{g(\varepsilon)} \dd x\\
		&\quad-\dfrac{1}{\varepsilon}\int_{\Omega} \dfrac{\left\{[\bm{\theta} + \varepsilon x_3 \bm{a}_3+u_{j,\kappa}(\varepsilon)\bm{g}^j(\varepsilon)]\cdot \bm{q}\right\}^{-}}{\kappa\sqrt{\sum_{\ell=1}^{3}|\bm{g}^\ell(\varepsilon) \cdot \bm{q}|^2}}[\varepsilon x_3(\bm{a}_3-(\partial_\alpha \eta_3 +b_\alpha^\sigma \eta_\sigma)\bm{a}^\alpha) \cdot \bm{q}] \sqrt{g(\varepsilon)} \dd x\\
		&\quad +\dfrac{1}{\varepsilon}\int_{\Omega} \dfrac{\left\{[\bm{\theta} + \varepsilon x_3 \bm{a}_3+u_{j,\kappa}(\varepsilon)\bm{g}^j(\varepsilon)]\cdot \bm{q}\right\}^{-}}{\kappa\sqrt{\sum_{\ell=1}^{3}|\bm{g}^\ell(\varepsilon) \cdot \bm{q}|^2}}\left([\bm{\theta} + \varepsilon x_3 \bm{a}_3+u_{i,\kappa}(\varepsilon)\bm{g}^i(\varepsilon)]\cdot \bm{q}\right) \sqrt{g(\varepsilon)} \dd x + O(\varepsilon)
	\end{aligned}
\end{equation}
has $\limsup$ less or equal than zero when $\varepsilon \to 0$ since the first addend in the last equality is less or equal than zero for all $\varepsilon>0$ being $\bm{\eta} \in \bm{U}_F(\omega)$, the second and the third addends tend to zero as $\varepsilon \to 0$ thanks to the third convergence of~\eqref{process1}, and the fourth addend is less or equal than zero thansk to the monotonicity of $-\{\cdot\}^{-}$ established in Lemma~\ref{lem:3}.
Note that the second factor in the second integral of the last equality corresponds to a remnant of the definition of $\bm{g}^\alpha$ (cf. Lemma~\ref{lem:2}).

The asymptotic behavior of the functions $\bm{w}(\varepsilon;\bm{\eta})$ and $\dfrac{1}{\varepsilon} e_{i\| j}(\varepsilon;\bm{w}(\varepsilon;\bm{\eta}))$ exhibited in~\eqref{weta}, the asymptotic behavior of the three-dimensional elasticity tensor $A^{ijk\ell}(\varepsilon)$ and $g(\varepsilon)$ (Lemma~\ref{lem:2}), the weak convergences $\dfrac{1}{\varepsilon}e_{i\| j}(\varepsilon) \rightharpoonup e_{i\|j}^1$ in $L^2(\Omega)$ established in part~(i), and the relations satisfied by $e_{i\| 3}^1$ (part~(iii)) together give:
\begin{equation}
\label{eqLe-1}
\begin{aligned}
\int_\Omega A^{ijk\ell}(\varepsilon)&\left\{\dfrac{1}{\varepsilon}e_{k\|\ell}(\varepsilon)\right\} \left\{\dfrac{1}{\varepsilon}e_{i\|j}(\varepsilon;\bm{w}(\varepsilon;\bm{\eta}))\right\} \sqrt{g(\varepsilon)} \dd x\\
&\to \int_\Omega (A^{\alpha\beta\sigma\tau}(0)e_{\sigma\|\tau}^1+A^{\alpha\beta33}(0)e_{3\|3}^1) \{-x_3 \rho_{\alpha\beta}(\bm{\eta})\} \sqrt{a} \dd x,\\
2 \int_\Omega A^{ijk\ell}(\varepsilon) &\left\{\dfrac{1}{\varepsilon} e_{k\|\ell}(\varepsilon)\right\} e_{i\|j}^1 \sqrt{g(\varepsilon)} \dd x
-\int_\Omega A^{ijk\ell}(\varepsilon) e_{k\|\ell}^1 e_{i\|j}^1 \sqrt{g(\varepsilon)} \dd x\\
&\to \int_\Omega A^{ijk\ell}(0) e_{k\|\ell}^1 e_{i\|j}^1 \sqrt{a} \dd x=\int_\Omega a^{\alpha\beta\sigma\tau} e_{\sigma\|\tau}^1 e_{\alpha\|\beta}^1 \sqrt{a} \dd x,\\
\int_\Omega f^i (w_i(\varepsilon;\bm{\eta}&)-u_{i,\kappa}(\varepsilon)) \sqrt{g(\varepsilon)} \dd x \to \int_\Omega f^i(\eta_i -u_i) \sqrt{a} \dd x=\int_\omega p^i (\eta_i -\overline{u}_i) \sqrt{a} \dd y.
\end{aligned}
\end{equation}
	
We have yet to take into account the relations $\rho_{\alpha \beta}(\bm{u})=-\partial_3 e_{\alpha\|\beta}^1$ in $L^2(\Omega)$ established in part~(iii). Since $\bm{u}$ is independent of $x_3$ (part (ii)), these relations show that the functions $e_{\alpha\|\beta}^1$ are of the form
	$$
	e_{\alpha\|\beta}^1=\Upsilon_{\alpha\beta}-x_3 \rho_{\alpha\beta}(\overline{\bm{u}})\quad\textup{ with }\Upsilon_{\alpha\beta} \in L^2(\omega).
	$$
	
The asymptotic behaviors observed in~\eqref{eqLe}--\eqref{eqLe-1}, Lemma~\ref{lem:1}, and the uniform positive-definiteness of the fourth order two-dimensional elasticity tensor $(a^{\alpha\beta\sigma\tau})$ in turn imply:
\begin{align*}
	0&\le \limsup_{\varepsilon \to 0} \Lambda(\varepsilon)
	\le \lim_{\varepsilon \to 0} \int_\Omega A^{ijk\ell}(\varepsilon)\left\{\dfrac{1}{\varepsilon}e_{k\|\ell}(\varepsilon)\right\} \left\{\dfrac{1}{\varepsilon}e_{i\|j}(\varepsilon;\bm{w}(\varepsilon;\bm{\eta}))\right\} \sqrt{g(\varepsilon)} \dd x\\
	&\quad -2 \lim_{\varepsilon \to 0} \int_\Omega A^{ijk\ell}(\varepsilon) \left\{\dfrac{1}{\varepsilon} e_{k\|\ell}(\varepsilon)\right\} e_{i\|j}^1 \sqrt{g(\varepsilon)} \dd x
	+\lim_{\varepsilon \to 0}\int_\Omega A^{ijk\ell}(\varepsilon) e_{k\|\ell}^1 e_{i\|j}^1 \sqrt{g(\varepsilon)} \dd x\\
	&\quad-\lim_{\varepsilon \to 0}\int_\Omega f^i (w_i(\varepsilon;\bm{\eta})-u_{i,\kappa}(\varepsilon)) \sqrt{g(\varepsilon)} \dd x\\
	&= -\int_\omega p^i (\eta_i-\overline{u}_i) \sqrt{a} \dd y -\int_\Omega A^{ijk\ell}(0) e_{k\|\ell}^1 e_{i\|j}^1 \sqrt{a} \dd x 
	+\int_\Omega A^{\alpha\beta k\ell}(0) e_{k\|\ell}^1 \{-x_3 \rho_{\alpha \beta}(\bm{\eta})\} \sqrt{a} \dd x\\
	&=-\int_\omega p^i (\eta_i-\overline{u}_i) \sqrt{a} \dd y -\dfrac{1}{2}\int_{\Omega} a^{\alpha\beta\sigma\tau} \Upsilon_{\sigma\tau}\Upsilon_{\alpha\beta} \sqrt{a} \dd x
	+\dfrac{1}{2}\int_\Omega a^{\alpha\beta\sigma\tau} x_3^2 \rho_{\sigma \tau}(\overline{\bm{u}}) \rho_{\alpha \beta}(\bm{\eta}-\overline{\bm{u}}) \sqrt{a} \dd x\\
	&\le \dfrac{1}{3} \int_\omega a^{\alpha\beta\sigma\tau} \rho_{\sigma \tau}(\overline{\bm{u}}) \rho_{\alpha \beta}(\bm{\eta}-\overline{\bm{u}})\dd y - \int_\omega p^i (\eta_i -\overline{u}_i) \sqrt{a} \dd y.
	\end{align*}

	In conclusion, the latter yields
	$$
	0 \le \dfrac{1}{3} \int_\omega a^{\alpha\beta\sigma\tau} \rho_{\sigma \tau}(\overline{\bm{u}}) \rho_{\alpha \beta}(\bm{\eta}-\overline{\bm{u}})\dd y - \int_\omega p^i (\eta_i -\overline{u}_i) \sqrt{a} \dd y,
	$$
	which establishes that $\overline{\bm{u}}$ is the unique solution for Problem~\ref{problemLim} since $\bm{\eta}=(\eta_i) \in \bm{U}_F(\omega)$ is arbitrarily chosen.
	
	(v) \emph{The weak convergence $\bm{u}_{\kappa}(\varepsilon) \rightharpoonup \bm{u}$ in $\bm{H}^1(\Omega)$ established in part~(i) is in fact strong, i.e.,
	$$
	\bm{u}_{\kappa}(\varepsilon) \to \bm{u} \quad\textup{ in }\bm{H}^1(\Omega),
	$$
	and holds for the whole sequence $(\bm{u}_{\kappa}(\varepsilon))_{\varepsilon>0}$.
	}
	
	The proof is identical to that of part~(vi) of the proof of Theorem~6.2-1 of~\cite{Ciarlet2000} and for this reason is omitted. The assertion means that for all $\delta>0$ there exists a number $\varepsilon_2=\varepsilon_2(\delta)>0$ such that for all $0<\varepsilon<\varepsilon_2$ it results
	\begin{equation*}
	\|\bm{u}_{\kappa}(\varepsilon) - \bm{u}\|_{\bm{H}^1(\Omega)} <\delta.
	\end{equation*}
	
	(vi) \emph{The weak convergences $\dfrac{1}{\varepsilon}e_{i\|j}(\varepsilon) \rightharpoonup e_{i\|j}^1$ in $L^2(\Omega)$ established in part \emph{(i)} are in fact strong, i.e.,
	$$
	\dfrac{1}{\varepsilon}e_{i\|j}(\varepsilon) \to e_{i\|j}^1 \quad\textup{ in }L^2(\Omega).
	$$
	Besides, the limits $e_{i\|j}^1$ are unique; hence these convergences hold for the whole family $\left(\dfrac{1}{\varepsilon}e_{i\|j}(\varepsilon)\right)_{\varepsilon>0}$.}

	Let $\Lambda(\varepsilon)$ be as in part (iv). Since $\overline{\bm{u}} \in \bm{U}_F(\omega)$, we define a vector field $\tilde{\bm{w}}(\varepsilon) = (\tilde{w}_i(\varepsilon)) \in \bm{V}(\Omega)$ as follows:
	\begin{align*}
	\tilde{w}_\alpha(\varepsilon)&:=\overline{u}_\alpha-\varepsilon x_3 (\partial_\alpha \overline{u}_3+b_\alpha^\sigma \overline{u}_\sigma),\\
	\tilde{w}_3(\varepsilon)&:=\overline{u}_3.
	\end{align*}

	We have that, in the same spirit as~\eqref{weta}:
	\begin{equation}
		\label{wu}
		\begin{aligned}
			\tilde{\bm{w}}(\varepsilon) \to& \overline{\bm{u}} \textup{ in }\bm{H}^1(\Omega),\\
			\dfrac{1}{\varepsilon} e_{\alpha\| \beta}(\varepsilon;\tilde{\bm{w}}(\varepsilon)) \to& \{-x_3 \rho_{\alpha \beta}(\overline{\bm{u}})\} \textup{ in } L^2(\Omega),\\
			\textup{The sequence }\left(\dfrac{1}{\varepsilon} e_{\alpha\| 3}(\varepsilon;\tilde{\bm{w}}(\varepsilon))\right)_{\varepsilon>0} &\textup{ converges in }L^2(\Omega).
		\end{aligned}
	\end{equation}

	Specializing $\bm{v}=(\tilde{\bm{w}}(\varepsilon)-\bm{u}_\kappa(\varepsilon))$ in the variational equations of Problem~\ref{problem1scaled} and repeating the same computations as in part~(iv) gives:
	$$
	0\le \limsup_{\varepsilon \to 0} \Lambda(\varepsilon)\le-\int_\omega a^{\alpha\beta\sigma\tau} \Upsilon_{\sigma\tau} \Upsilon_{\alpha\beta} \sqrt{a} \dd y \le 0.
	$$
	
In conclusion, by the uniform positive-definiteness of the two-dimensional fourth-order elasticity tensor (Lemma~\ref{lem:2}) and the asymptotic behavior of the sequence of vector fields $(\tilde{\bm{w}}(\varepsilon))_{\varepsilon>0}$ in~\eqref{wu}, we obtain $\Upsilon_{\alpha\beta}=0$ and that
	$$
	\lim_{\varepsilon \to 0} \Lambda(\varepsilon)=0.
	$$
	
	These relations in turn imply that the strong convergence
	$$
	\dfrac{1}{\varepsilon} e_{i\|j}(\varepsilon) \to e_{i\|j}^1\quad\textup{ in }L^2(\Omega)
	$$
	holds. The functions $e_{\alpha\|\beta}^1$ are uniquely determined, since they are given by
	$$
	e_{\alpha\|\beta}^1=-x_3 \rho_{\alpha\beta}(\overline{\bm{u}})
	$$
	and the vector field $\overline{\bm{u}}$ is uniquely determined as the unique solution of Problem~\ref{problemLim}. That the functions $e_{i\|3}^1$ are uniquely determined then follows from the relations established in part~(iii). Therefore, the whole sequence $(\varepsilon^{-1} e_{i\|j}(\varepsilon))_{\varepsilon>0}$ strongly converges to the functions $e_{i\|j}^1$ in $L^2(\Omega)$ and the proof is complete.
\end{proof}

Observe that the conclusion of Theorem~\ref{t:5} and Theorem~4.2-1(b) of~\cite{Ciarlet2000} give that for each $\delta>0$ and for each $1 \le i \le 3$ there exists a number $\varepsilon_2=\varepsilon_2(\delta)>0$ such that for all $0<\varepsilon<\varepsilon_2$ it results
\begin{equation}
	\label{lim:2}
	\left\|\dfrac{1}{2}\int_{-1}^{1} u_{i,\kappa}(\varepsilon)(\cdot,x_3) \dd x_3 - \zeta_i\right\|_{H^1(\omega)} < \dfrac{\delta}{2},
\end{equation}
where $\bm{u}_\kappa(\varepsilon)=(u_{i,\kappa}(\varepsilon))$ is the unique solution of Problem~\ref{problem1scaled}, and $\bm{\zeta}=(\zeta_i)$ is the unique solution of Problem~\ref{problemLim}, and $\kappa$ is as in~\eqref{kappa}.

Therefore, the convergence as $\varepsilon \to 0$ of the solutions $\bm{u}(\varepsilon)$ of Problem~\ref{problem0scaled} to the solution $\bm{\zeta}$ of Problem~\ref{problemLim} can thus be established as a direct corollary of Theorem~\ref{th:beta-6} and Theorem~\ref{t:5}.

\begin{corollary}
\label{cor1}
Let $\omega$ be a domain in $\mathbb{R}^2$, let $\bm{\theta} \in \mathcal{C}^3(\overline{\omega}; \mathbb{E}^3)$ be the middle surface of a \emph{flexural shell}, let $\gamma_0$ be a non-zero length portion of the boundary $\gamma$ (cf. section~\ref{Sec:3}) and let $\bm{q} \in \mathbb{E}^3$ be a non-zero vector given once and for all. Let us consider the \emph{non-trivial} space (cf. section~\ref{Sec:3})
$$
\bm{V}_F(\omega):=\{\bm{\eta}=(\eta_i) \in H^1(\omega)\times H^1(\omega)\times H^2(\omega); \gamma_{\alpha \beta}(\bm{\eta})=0 \textup{ in }\omega \textup{ and }\eta_i=\partial_{\nu}\eta_3=0 \textup{ on }\gamma_0\},
$$
and let us define the set
\begin{align*}
	\bm{U}_F (\omega) &:= \{\bm{\eta} = (\eta_i) \in \bm{V}_F(\omega);\big(\bm{\theta} (y) + \eta_i (y) \bm{a}^i(y)\big) \cdot \bm{q} \ge 0 \textup{ for a.a. } y \in \omega \}.
\end{align*}
%and assume that the immersion $\bm{\theta}$ is such that
%\[
%d:= \inf_{y\in \overline{\omega}} (\bm{\theta} (y) \cdot \bm{q}) > 0.
%\]

Let there be given a family of linearly elastic flexural shells with the same middle surface $\bm{\theta} (\overline{\omega})$ and thickness $2 \varepsilon > 0$, and let $\bm{u}(\varepsilon) \in \bm{U}(\varepsilon;\Omega)$ denote, for each $\varepsilon > 0$, the unique solution of Problem~\ref{problem0scaled}.

Then, we have that
$$
\dfrac{1}{2} \int_{-1}^{1} \bm{u}(\varepsilon) \dd x_3 \to \bm{\zeta},\quad \textup{ in } \bm{H}^1(\omega) \textup{ as } \varepsilon\to 0,
$$
where $\bm{\zeta}$ is the unique solution to the two-dimensional variational Problem~\ref{problemLim}.	
\end{corollary}
\begin{proof}
To prove the claim, we have to show that for each $\delta>0$ and for all $1 \le i \le 3$ there exists a number $\tilde{\varepsilon}=\tilde{\varepsilon}(\delta)>0$ such that for all $0<\varepsilon<\tilde{\varepsilon}$ it results
\begin{equation}
\label{lim:3}
\left\|\dfrac{1}{2} \int_{-1}^{1} u_i(\varepsilon)(\cdot, x_3) \dd x_3 - \zeta_i\right\|_{H^1(\omega)} < \delta,
\end{equation}
where $\bm{u}(\varepsilon)=(u_i(\varepsilon))$ is the unique solution of Problem~\ref{problem0scaled} and $\bm{\zeta}=(\zeta_i)$ is the unique solution of Problem~\ref{problemLim}.

In order to prove~\eqref{lim:3}, fix $\delta>0$ and take $\tilde{\varepsilon}:=\varepsilon_2=\varepsilon_2(\delta)$. For each $0<\varepsilon <\varepsilon_2$ we have that an application of~\eqref{lim:1}, the triangle inequality, Theorem~4.2-1 (b) of~\cite{Ciarlet2000}, and~\eqref{lim:2} gives
\begin{equation}
\label{lim:4}
\begin{aligned}
&\left\|\dfrac{1}{2} \int_{-1}^{1} u_i(\varepsilon) \dd x_3 - \zeta_i\right\|_{H^1(\omega)}
\le \left\|\dfrac{1}{2}\int_{-1}^{1} u_{i,\kappa}(\varepsilon)(\cdot,x_3) \dd x_3 - \zeta_i\right\|_{H^1(\omega)}
+\dfrac{1}{\sqrt{2}}\|u_i(\varepsilon)-u_{i,\kappa}(\varepsilon)\|_{H^1(\Omega)}\\
&< \dfrac{\delta}{2}+\dfrac{\delta}{2}=\delta,
\end{aligned}
\end{equation}
whenever $0<\kappa < \kappa_0(\delta)$ (viz. Theorem~\ref{th:beta-6} for a reasonable sufficient condition ensuring that $\kappa_0$ is independent of $\varepsilon$) and $\kappa$ is as in~\eqref{kappa}. 

Observe that if we take $\kappa$ as in~\eqref{kappa} and it resulted $\kappa_0(\delta)<\sqrt{\varepsilon}=\kappa$, then it would suffice to reduce $\varepsilon_2$ until the requirement for the penalty parameter is met. In doing so, the fact that $\kappa_0$ is independent of $\varepsilon$ plays a critical role in establishing the sought convergence.

The estimate~\eqref{lim:4} means that the average across the thickness of the solutions $\bm{u}(\varepsilon)$ of Problem~\ref{problem0scaled} converge to the solution of Problem~\ref{problemLim} as $\varepsilon \to 0$, as it had to be proved.
\end{proof}

\begin{figure}[H]
\centering
\includegraphics[width=0.5\textwidth]{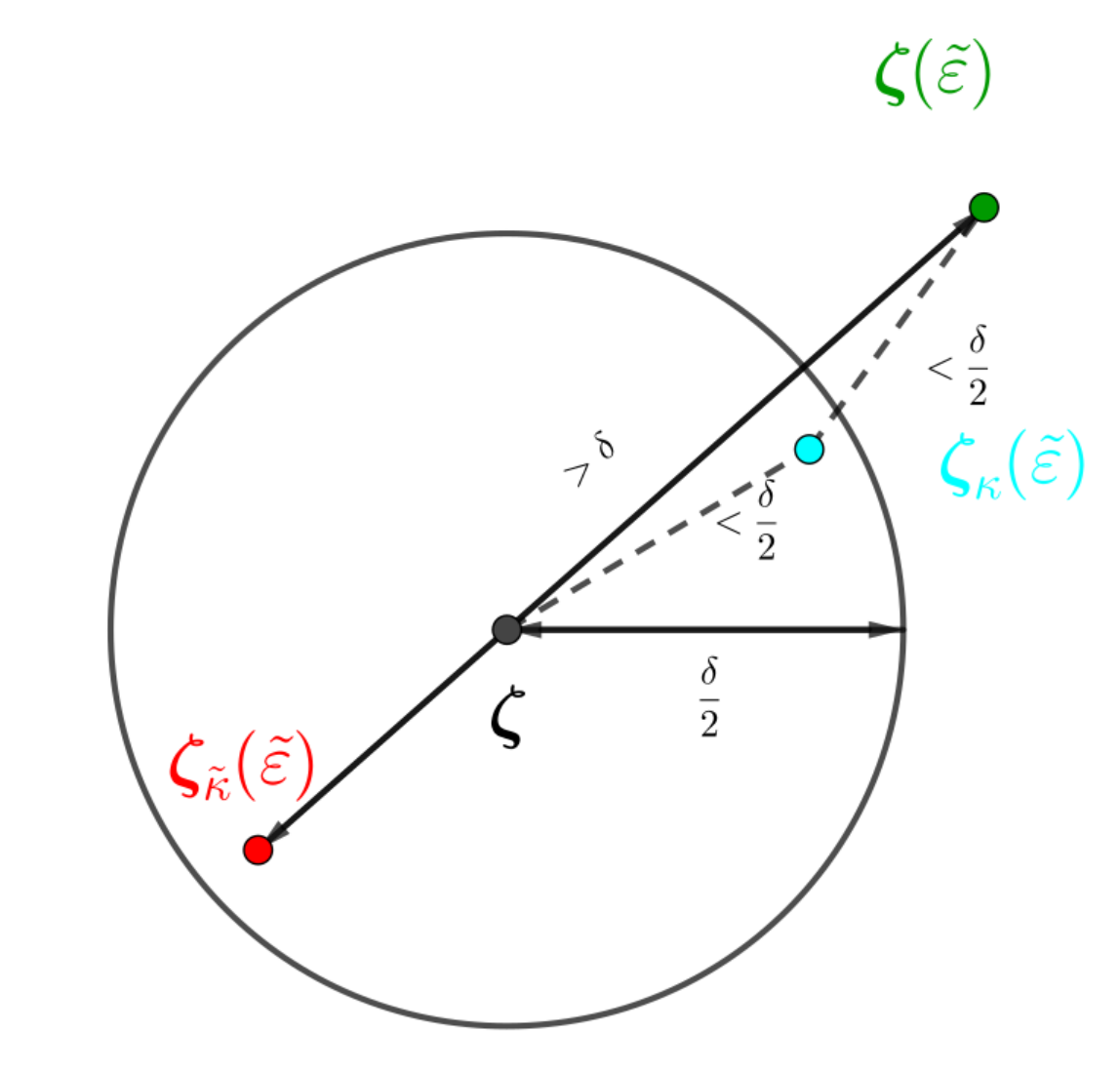}
\caption{\scriptsize \emph{Geometrical interpretation of Corollary~\ref{cor1} and criticality of the condition~\eqref{kappa}}. 
	Let $\delta>0$ be fixed and let $0<\tilde{\varepsilon}<\varepsilon_2$, where $\varepsilon_2$ is the threshold parameter found in Corollary~5.2.
    Let $\bm{\zeta}$ denote the solution to Problem~\ref{problemLim}. 
    Let $\tilde{\kappa}$ be a ``penalty'' parameter such that $\tilde{\kappa}:=\sqrt{\tilde{\varepsilon}}$, and for which the distance between the averaged solution of Problem~\hyperref[problem1scaled]{$\mathcal{P}_{\tilde{\kappa}}(\tilde{\varepsilon};\Omega)$} (i.e., the vector field $\bm{\zeta}_{\tilde{\kappa}}(\tilde{\varepsilon})$) and the solution $\bm{\zeta}$ of Problem~\ref{problemLim} with respect to the $\bm{H}^1(\omega)$ norm is less than $\delta/2$.
	\newline
	If $\tilde{\kappa}=\sqrt{\tilde{\varepsilon}} >\kappa_0=\kappa_0(\delta)$, it \emph{might} happen, as a consequence of~\eqref{lim:1}, that the distance between the averaged solution of Problem~\hyperref[problem1scaled]{$\mathcal{P}_{\tilde{\kappa}}(\tilde{\varepsilon};\Omega)$} (i.e., the vector field $\bm{\zeta}_{\tilde{\kappa}}(\tilde{\varepsilon})$) and the averaged solution of Problem~\hyperref[problem0scaled]{$\mathcal{P}(\tilde{\varepsilon};\Omega)$} (i.e., the vector field $\bm{\zeta}(\tilde{\varepsilon})$) is, with respect to the $\bm{H}^1(\omega)$ norm, strictly greater than the fixed number $\delta$.
	\newline
	If, however, we further reduce the parameter $\tilde{\varepsilon}$ until $\tilde{\kappa}<\kappa_0$, then we can find an element $\bm{\zeta}_{\kappa}(\tilde{\varepsilon})$, solution of Problem~\hyperref[problem1scaled]{$\mathcal{P}_\kappa(\tilde{\varepsilon};\Omega)$}, whose distance (once again with respect to the $\bm{H}^1(\omega)$ norm) from both $\bm{\zeta}$ and $\bm{\zeta}(\tilde{\varepsilon})$ is strictly less than $\delta/2$. In the latter step, a critical role is played by the fact that $\kappa_0$ is assumed to be independent of $\varepsilon$ (viz. Theorem~\ref{th:beta-6} for a reasonable sufficient condition). This figure originally appeared in~\cite{PieJDE2022}.
}
\label{fig:1}
\end{figure}

It remains to ``\emph{de-scale}'' the results of Theorem~\ref{t:5} and Corollary~\ref{cor1}, which apply to the solutions $\bm{u}(\varepsilon)$ of the \emph{scaled} problem~\ref{problem0scaled}. This means that we need to convert these results into ones about the \emph{unknown} $u^\varepsilon_i \bm{g}^{i, \varepsilon} : \overline{\Omega^\varepsilon} \to \mathbb{E}^3$, which represents the physical \emph{three-dimensional vector field} of the actual reference configuration of the shell. As shown in the next theorem, this conversion is most conveniently achieved through the introduction of the \emph{averages} $\displaystyle \frac{1}{2\varepsilon} \int^\varepsilon_{-\varepsilon} u^\varepsilon_i \bm{g}^{i, \varepsilon} \dd x^\varepsilon_3$ across the thickness of the shell.

\begin{theorem} \label{t:6}
Let the assumptions on the data be as in \emph{section~\ref{Sec:3}} and let the assumptions on the immersion $\bm{\theta} \in \mathcal{C}^3 (\overline{\omega}; \mathbb{E}^3)$ be as in \emph{Theorem \ref{t:4}}. Let $\bm{u}^\varepsilon = (u^\varepsilon_i) \in \bm{U}(\Omega^\varepsilon)$ denote for each $\varepsilon > 0$ the unique solution of the variational Problem~\ref{problem0} and let $\bm{\zeta} \in \bm{U}_F(\omega)$ denote the unique solution to the variational inequalities in Problem~\ref{problemLim}. Then
\begin{align*}
  \frac{1}{2\varepsilon} \int^\varepsilon_{-\varepsilon} u^\varepsilon_\alpha \bm{g}^{\alpha , \varepsilon} \dd x^\varepsilon_3 & \to \zeta_\alpha \bm{a}^\alpha, \quad \textup{ in } \bm{H}^1 (\omega) \textup{ as } \varepsilon \to 0, \\
  \frac{1}{2\varepsilon} \int^\varepsilon_{-\varepsilon} u^\varepsilon_3 \bm{g}^{3 , \varepsilon} \dd x^\varepsilon_3 &\to \zeta_3 \bm{a}^3, \quad
\textup{ in } \bm{H}^1 (\omega) \textup{ as } \varepsilon \to 0.
\end{align*}
\end{theorem}

\begin{proof}
The proof is analogous to that of Theorem~6.4-1 in~\cite{Ciarlet2000} and for this reason is omitted.
\end{proof}

In view of the scalings, viz., $u^\varepsilon_i (x^\varepsilon) = u_i (\varepsilon) (x)$ at each $x^\varepsilon \in \Omega^\varepsilon$, and of the assumption  on the data, viz., $f^{i, \varepsilon} (x^\varepsilon) = f^i(x)$ at each $x^\varepsilon \in \Omega^\varepsilon$, made in section~\ref{Sec:3}, it is natural to also ``de-scale'' the unknown appearing in the \emph{limit two-dimensional problem} found in Theorem \ref{t:4}, by letting
\[
  \zeta^\varepsilon_i (y) := \zeta_i (y) \textup{ at each } y \in \omega,
\]
and by using in its formulation the contravariant components
\[
p^{i , \varepsilon} := \varepsilon^3 \left\{\int_{-1}^{1}f^i \dd x_3\right\}=\int^\varepsilon_{-\varepsilon} f^{i, \varepsilon}\dd x^\varepsilon_3
\]
instead of their scaled counterparts $p^i= \int^1_{-1} f^i \dd x_3$. In this fashion, it is  immediately found that $\bm{\zeta}^\varepsilon = (\zeta^\varepsilon_i) \in \bm{U}_F (\omega)$ \emph{is the unique solution to the variational inequalities}
\begin{align*}
\dfrac{\varepsilon^3}{3} \int_\omega a^{\alpha \beta \sigma \tau} \rho_{\sigma \tau} (\bm{\zeta}^\varepsilon
) \rho_{\alpha \beta} (\bm{\eta}- \bm{\zeta}^\varepsilon) \sqrt{a} \dd y \ge \int_\omega p^{i, \varepsilon} (\eta_i - \zeta^\varepsilon_i) \sqrt{a} \dd y,
\end{align*}
for all $\bm{\eta} = (\eta_i) \in \bm{U}_F (\omega)$. These inequalities now display the \emph{factor} $\dfrac{\varepsilon^3}{3}$, which always appears in the left-hand sides of equations modelling \emph{flexural} shells.

\section*{Conclusion and final remarks}

In this paper we identified a set of two-dimensional variational inequalities that model the displacement of a linearly elastic flexural shell subjected to a confinement condition, expressing that all the points of the admissible deformed configurations remain in a given half-space.

The starting point of the rigorous asymptotic analysis we carried out is a set of variational inequalities based on the classical equations of three-dimensional linearized elasticity, and posed over a non-empty, closed and convex subset of a suitable Sobolev space. These variational inequalities govern the displacement of a three-dimensional flexural shell subject to a confinement condition like the one recalled beforehand.

By means of the penalized version of the aforementioned problem (i.e., the set of variational inequalities based on the classical equations of three-dimensional linearized elasticity), we managed to show that, as the thickness parameter approaches zero, the average across the thickness of the solution of the original three-dimensional model converges to the solution of a \emph{ad hoc} two dimensional model. In this regard, it is worth recalling that the rigorous asymptotic analysis (Theorem~\ref{t:5}) hinged on an \emph{ad hoc} scaling of the penalty coefficient with respect to the shell thickness as well as the critical assumption~\eqref{kappa}.

The two-dimensional model we recovered in this paper coincides with the two-dimensional model recovered as a result of a rigorous asymptotic analysis carried out starting from Koiter's model in the case where the linearly elastic shell under consideration is subjected to an obstacle (viz. \cite{CiaPie2018b} and~\cite{CiaPie2018bCR}).

It is worth mentioning that, unlike the case where the linearly elastic shell under consideration was a linearly elastic elliptic membrane shell (viz. \cite{CiaMarPie2018b}, \cite{CiaMarPie2018} and~\cite{PieJDE2022}), the rigorous asymptotic analysis carried out in this paper does not resort to any additional assumption apart from those already made in~\cite{CiaLodsMia1996} (see also Chapter~6 of~\cite{Ciarlet2000}).

Indeed, when the linearly elastic shell under consideration was a linearly elastic elliptic membrane shell, the authors had to resort to a \emph{ad hoc} ``density property'' in order to carry out the asymptotic analysis leading to the identification of a suitable two-dimensional set of variational inequalities.
The ``density property'' we mentioned is ensured only under certain geometrical assumptions, which restrict the applicability of the recalled result. This very ``density property'' was recently exploited in~\cite{Pie-2021-2} (see also~\cite{Pie-2021-1}) to show that the solution of a certain obstacle problem in linearized elasticity enjoys, at least locally, higher regularity properties.

\section*{Acknowledgements}
The author is greatly indebted to Professor Philippe G. Ciarlet for his encouragement and guidance.

The author would like to express his sincere gratitude to the Anonymous Referees for their suggested improvements.

The author was partly supported by the Research Fund of Indiana University.

The author declares that no experimental data was used in the preparation of this manuscript.

The author declares that there is no conflict of interests.

%% file: Pie_Flex_Shells_2022_R1.bbl
\begin{thebibliography}{51}
\providecommand{\natexlab}[1]{#1}
\providecommand{\url}[1]{\texttt{#1}}
\expandafter\ifx\csname urlstyle\endcsname\relax
  \providecommand{\doi}[1]{doi: #1}\else
  \providecommand{\doi}{doi: \begingroup \urlstyle{rm}\Url}\fi

\bibitem[Agmon et~al.(1959)Agmon, Douglis, and Nirenberg]{AgmDouNir1959}
S.~Agmon, A.~Douglis, and L.~Nirenberg.
\newblock Estimates near the boundary for solutions of elliptic partial
  differential equations satisfying general boundary conditions. {I}.
\newblock \emph{Comm. Pure Appl. Math.}, 12:\penalty0 623--727, 1959.

\bibitem[Agmon et~al.(1964)Agmon, Douglis, and Nirenberg]{AgmDouNir1964}
S.~Agmon, A.~Douglis, and L.~Nirenberg.
\newblock Estimates near the boundary for solutions of elliptic partial
  differential equations satisfying general boundary conditions. {II}.
\newblock \emph{Comm. Pure Appl. Math.}, 17:\penalty0 35--92, 1964.

\bibitem[Bernadou and Ciarlet(1976)]{BerCia1976}
M.~Bernadou and P.~G. Ciarlet.
\newblock Sur l'ellipticit\'{e} du mod\`ele lin\'{e}aire de coques de {W}. {T}.
  {K}oiter.
\newblock \emph{Computing methods in applied sciences and engineering ({S}econd
  {I}nternat. {S}ympos., {V}ersailles, 1975), {P}art 1. Lecture Notes in
  Econom. and Math. Systems}, 134:\penalty0 89--136, 1976.

\bibitem[Bernadou et~al.(1994)Bernadou, Ciarlet, and Miara]{BerCiaMia1994}
M.~Bernadou, P.~G. Ciarlet, and B.~Miara.
\newblock Existence theorems for two-dimensional linear shell theories.
\newblock \emph{{J}. {E}lasticity}, 34:\penalty0 111--138, 1994.

\bibitem[Bernard(2011)]{Ber2011}
J.~M.~E. Bernard.
\newblock Density results on {S}obolev spaces whose elements vanish on a part
  of the boundary.
\newblock \emph{Chinese Ann. Math., Ser B}, 32:\penalty0 823--846, 2011.

\bibitem[Brezis and Stampacchia(1968)]{BrezisStampacchia1968}
H.~Brezis and G.~Stampacchia.
\newblock Sur la r\'{e}gularit\'{e} de la solution d'in\'{e}quations
  elliptiques.
\newblock \emph{Bull. Soc. Math. France}, 96:\penalty0 153--180, 1968.

\bibitem[Caillerie and
  S\'{a}nchez-Palencia(1995{\natexlab{a}})]{CaillerieSanchez1995a}
D.~Caillerie and E.~S\'{a}nchez-Palencia.
\newblock A new kind of singular stiff problems and application to thin elastic
  shells.
\newblock \emph{Math. Models Methods Appl. Sci.}, 5\penalty0 (1):\penalty0
  47--66, 1995{\natexlab{a}}.

\bibitem[Caillerie and
  S\'{a}nchez-Palencia(1995{\natexlab{b}})]{CaillerieSanchez1995b}
D.~Caillerie and E.~S\'{a}nchez-Palencia.
\newblock Elastic thin shells: asymptotic theory in the anisotropic and
  heterogeneous cases.
\newblock \emph{Math. Models Methods Appl. Sci.}, 5\penalty0 (4):\penalty0
  473--496, 1995{\natexlab{b}}.

\bibitem[Ciarlet(1988)]{Ciarlet1988}
P.~G. Ciarlet.
\newblock \emph{Mathematical Elasticity. Vol. I: Three-Dimensional Elasticity}.
\newblock North-Holland, Amsterdam, 1988.

\bibitem[Ciarlet(1989)]{PGCNLAO}
P.~G. Ciarlet.
\newblock \emph{Introduction to Numerical Linear Algebra and Optimisation}.
\newblock Cambridge University Press, 1989.

\bibitem[Ciarlet(2000)]{Ciarlet2000}
P.~G. Ciarlet.
\newblock \emph{Mathematical Elasticity. Vol. III: {T}heory of Shells}.
\newblock North-Holland, Amsterdam, 2000.

\bibitem[Ciarlet(2005)]{Ciarlet2005}
P.~G. Ciarlet.
\newblock \emph{An {I}ntroduction to {D}ifferential {G}eometry with
  {A}pplications to {E}lasticity}.
\newblock Springer, Dordrecht, 2005.

\bibitem[Ciarlet(2013)]{PGCLNFAA}
P.~G. Ciarlet.
\newblock \emph{Linear and Nonlinear Functional Analysis with Applications}.
\newblock Society for Industrial and Applied Mathematics, Philadelphia, 2013.

\bibitem[Ciarlet and Destuynder(1979)]{CiaDes1979}
P.~G. Ciarlet and P.~Destuynder.
\newblock A justification of the two-dimensional linear plate model.
\newblock \emph{{J}. {M}\'ecanique}, 18:\penalty0 315--344, 1979.

\bibitem[Ciarlet and Lods(1996{\natexlab{a}})]{CiaLods1996a}
P.~G. Ciarlet and V.~Lods.
\newblock On the ellipticity of linear membrane shell equations.
\newblock \emph{{J}. {M}ath. {P}ures {A}ppl.}, 75:\penalty0 107--124,
  1996{\natexlab{a}}.

\bibitem[Ciarlet and Lods(1996{\natexlab{b}})]{CiaLods1996b}
P.~G. Ciarlet and V.~Lods.
\newblock Asymptotic analysis of linearly elastic shells. {I}. {J}ustification
  of membrane shell equations.
\newblock \emph{{A}rch. {R}ational {M}ech. {A}nal.}, 136\penalty0 (2):\penalty0
  119--161, 1996{\natexlab{b}}.

\bibitem[Ciarlet and Lods(1996{\natexlab{c}})]{CiaLods1996c}
P.~G. Ciarlet and V.~Lods.
\newblock Asymptotic analysis of linearly elastic shells. {III}.
  {J}ustification of koiter's shell equations.
\newblock \emph{{A}rch. {R}ational {M}ech. {A}nal.}, 136\penalty0 (2):\penalty0
  191--200, 1996{\natexlab{c}}.

\bibitem[Ciarlet and Lods(1996{\natexlab{d}})]{CiaLods1996d}
P.~G. Ciarlet and V.~Lods.
\newblock Asymptotic analysis of linearly elastic shells: ``generalized
  membrane shells''.
\newblock \emph{{J}. {E}lasticity}, 43\penalty0 (2):\penalty0 147--188,
  1996{\natexlab{d}}.

\bibitem[Ciarlet and Piersanti(2019{\natexlab{a}})]{CiaPie2018b}
P.~G. Ciarlet and P.~Piersanti.
\newblock Obstacle problems for {K}oiter's shells.
\newblock \emph{{M}ath. {M}ech. {S}olids}, 24:\penalty0 3061--3079,
  2019{\natexlab{a}}.

\bibitem[Ciarlet and Piersanti(2019{\natexlab{b}})]{CiaPie2018bCR}
P.~G. Ciarlet and P.~Piersanti.
\newblock A confinement problem for a linearly elastic {K}oiter's shell.
\newblock \emph{C.R. Acad. Sci. Paris, S\'{e}r. I}, 357:\penalty0 221--230,
  2019{\natexlab{b}}.

\bibitem[Ciarlet et~al.(1996)Ciarlet, Lods, and Miara]{CiaLodsMia1996}
P.~G. Ciarlet, V.~Lods, and B.~Miara.
\newblock Asymptotic analysis of linearly elastic shells. {II}. {J}ustification
  of flexural shell equations.
\newblock \emph{Arch. {R}ational {M}ech. {A}nal.}, 136\penalty0 (2):\penalty0
  163--190, 1996.

\bibitem[Ciarlet et~al.(2018)Ciarlet, Mardare, and Piersanti]{CiaMarPie2018b}
P.~G. Ciarlet, C.~Mardare, and P.~Piersanti.
\newblock Un probl\`eme de confinement pour une coque membranaire
  lin\'eairement \'elastique de type elliptique.
\newblock \emph{C.R. Acad. Sci. Paris, S\'{e}r. I}, 356\penalty0 (10):\penalty0
  1040--1051, 2018.

\bibitem[Ciarlet et~al.(2019)Ciarlet, Mardare, and Piersanti]{CiaMarPie2018}
P.~G. Ciarlet, C.~Mardare, and P.~Piersanti.
\newblock An obstacle problem for elliptic membrane shells.
\newblock \emph{{M}ath. {M}ech. {S}olids}, 24\penalty0 (5):\penalty0
  1503--1529, 2019.

\bibitem[Duvaut and Lions(1976)]{DuvLio76}
G.~Duvaut and J.~L. Lions.
\newblock \emph{Inequalities in {M}echanics and {P}hysics}.
\newblock Springer, Berlin, 1976.

\bibitem[Glowinski(1984)]{Glow84}
R.~Glowinski.
\newblock \emph{Numerical {M}ethods for {N}onlinear {V}ariational {P}roblems}.
\newblock Springer-Verlag, New York, 1984.

\bibitem[Kikuchi and Oden(1988)]{KikuchiOden1988}
N.~Kikuchi and J.~T. Oden.
\newblock \emph{Contact problems in elasticity: a study of variational
  inequalities and finite element methods}, volume~8.
\newblock Society for Industrial and Applied Mathematics (SIAM), Philadelphia,
  PA, 1988.

\bibitem[L\'eger and Miara(2008)]{LegMia2008}
A.~L\'eger and B.~Miara.
\newblock Mathematical justification of the obstacle problem in the case of a
  shallow shell.
\newblock \emph{J. {E}lasticity}, 90:\penalty0 241--257, 2008.

\bibitem[L\'eger and Miara(2018)]{LegMia2018}
A.~L\'eger and B.~Miara.
\newblock A linearly elastic shell over an obstacle: The flexural case.
\newblock \emph{{J}. {E}lasticity}, 131:\penalty0 19--38, 2018.

\bibitem[Lions(1969)]{Lions1969}
J.~L. Lions.
\newblock \emph{Quelques m\'{e}thodes de r\'{e}solution des probl\`emes aux
  limites non lin\'{e}aires}.
\newblock Dunod; Gauthier-Villars, Paris, 1969.

\bibitem[Meixner and Piersanti(Submitted)]{MeiPie2022}
A.~Meixner and P.~Piersanti.
\newblock {N}umerical approximation of the solution of an obstacle problem
  modelling the displacement of elliptic membrane shells via the penalty
  method.
\newblock Submitted.
\newblock URL \url{https://arxiv.org/abs/2205.11293}.

\bibitem[Miara and Sanchez-Palencia(1996)]{MiaSan1996}
B.~Miara and E.~Sanchez-Palencia.
\newblock Asymptotic analysis of linearly elastic shells.
\newblock \emph{Asymptotic Anal.}, 12\penalty0 (1):\penalty0 41--54, 1996.

\bibitem[Piersanti(2020)]{PieMMS2020}
P.~Piersanti.
\newblock An existence and uniqueness theorem for the dynamics of flexural
  shells.
\newblock \emph{Math. Mech. Solids}, 25\penalty0 (2):\penalty0 317--336, 2020.

\bibitem[Piersanti(2021)]{Pie-2021}
P.~Piersanti.
\newblock On the justification of the frictionless time-dependent {K}oiter's
  model for thermoelastic shells.
\newblock \emph{J. Differential Equations}, 296:\penalty0 50--106, 2021.

\bibitem[Piersanti(2022{\natexlab{a}})]{Pie-2021-1}
P.~Piersanti.
\newblock On the improved interior regularity of the solution of a fourth order
  elliptic problem modelling the displacement of a linearly elastic shallow
  shell lying subject to an obstacle.
\newblock \emph{Asymptot. Anal.}, 127\penalty0 (1-2):\penalty0 35--55,
  2022{\natexlab{a}}.

\bibitem[Piersanti(2022{\natexlab{b}})]{Pie-2021-2}
P.~Piersanti.
\newblock On the improved interior regularity of the solution of a second order
  elliptic boundary value problem modelling the displacement of a linearly
  elastic elliptic membrane shell subject to an obstacle.
\newblock \emph{Discrete Contin. Dyn. Syst.}, 42\penalty0 (2):\penalty0
  1011--1037, 2022{\natexlab{b}}.

\bibitem[Piersanti(2022{\natexlab{c}})]{PieJDE2022}
P.~Piersanti.
\newblock Asymptotic analysis of linearly elastic elliptic membrane shells
  subjected to an obstacle.
\newblock \emph{J. {D}ifferential {E}quations}, 320:\penalty0 114--142,
  2022{\natexlab{c}}.

\bibitem[Piersanti and Temam(2023)]{PieTem2023}
P.~Piersanti and R.~Temam.
\newblock On the dynamics of grounded shallow ice sheets: Modelling and
  analysis.
\newblock \emph{Adv. Nonlinear Anal.}, 12\penalty0 (1):\penalty0 40 pp., 2023.

\bibitem[Piersanti et~al.(2022{\natexlab{a}})Piersanti, White, Dragnea, and
  Temam]{PWDT2D}
P.~Piersanti, K.~White, B.~Dragnea, and R.~Temam.
\newblock Modelling virus contact mechanics under atomic force imaging
  conditions.
\newblock \emph{Appl. Anal.}, 101\penalty0 (11):\penalty0 3947--3957,
  2022{\natexlab{a}}.

\bibitem[Piersanti et~al.(2022{\natexlab{b}})Piersanti, White, Dragnea, and
  Temam]{PWDT3D}
P.~Piersanti, K.~White, B.~Dragnea, and R.~Temam.
\newblock A three-dimensional discrete model for approximating the deformation
  of a viral capsid subjected to lying over a flat surface in the static and
  time-dependent case.
\newblock \emph{{A}nal. {A}ppl. (Singap.)}, 20\penalty0 (6):\penalty0
  1159--1191, 2022{\natexlab{b}}.

\bibitem[Piersanti et~al.(2021)Piersanti, Africa, Fedele, Vergara, Ded\`{e},
  Corno, and Quarteroni]{Quarteroni2021-3}
R.~Piersanti, P.~C. Africa, M.~Fedele, C.~Vergara, L.~Ded\`{e}, A.~F. Corno,
  and A.~Quarteroni.
\newblock Modeling cardiac muscle fibers in ventricular and atrial
  electrophysiology simulations.
\newblock \emph{Comput. Methods Appl. Mech. Engrg.}, 373:\penalty0 113468, 33,
  2021.

\bibitem[Regazzoni et~al.(2021)Regazzoni, Ded\`{e}, and
  Quarteroni]{Quarteroni2021-2}
F.~Regazzoni, L.~Ded\`{e}, and A.~Quarteroni.
\newblock Active force generation in cardiac muscle cells: mathematical
  modeling and numerical simulation of the actin-myosin interaction.
\newblock \emph{Vietnam J. Math.}, 49\penalty0 (1):\penalty0 87--118, 2021.

\bibitem[Rodr\'{\i}guez-Ar\'{o}s(2018)]{Rodri2018}
A.~Rodr\'{\i}guez-Ar\'{o}s.
\newblock Mathematical justification of the obstacle problem for elastic
  elliptic membrane shells.
\newblock \emph{{A}pplicable {A}nal.}, 97:\penalty0 1261--1280, 2018.

\bibitem[Scholz(1984)]{Scholz1986}
R.~Scholz.
\newblock Numerical solution of the obstacle problem by the penalty method.
\newblock \emph{Computing}, 32\penalty0 (4):\penalty0 297--306, 1984.

\bibitem[Shen and Li(2018)]{Shen2018}
X.~Shen and H.~Li.
\newblock The time-dependent {K}oiter model and its numerical computation.
\newblock \emph{Appl. Math. Model.}, 55:\penalty0 131--144, 2018.

\bibitem[Shen et~al.(2019)Shen, Jia, Zhu, Li, Bai, Wang, and Cao]{Shen2019}
X.~Shen, J.~Jia, S.~Zhu, H.~Li, L.~Bai, T.~Wang, and X.~Cao.
\newblock The time-dependent generalized membrane shell model and its numerical
  computation.
\newblock \emph{Comput. Methods Appl. Mech. Engrg.}, 344:\penalty0 54--70,
  2019.

\bibitem[Shen et~al.(2020{\natexlab{a}})Shen, Piersanti, and
  Piersanti]{ShenPiePie2020}
X.~Shen, L.~Piersanti, and P.~Piersanti.
\newblock Numerical simulations for the dynamics of flexural shells.
\newblock \emph{Math. Mech. Solids}, 25\penalty0 (4):\penalty0 887--912,
  2020{\natexlab{a}}.

\bibitem[Shen et~al.(2020{\natexlab{b}})Shen, Yang, Li, Gao, and
  Wang]{Shen2020}
X.~Shen, Q.~Yang, L.~Li, Z.~Gao, and T.~Wang.
\newblock Numerical approximation of the dynamic {K}oiter's model for the
  hyperbolic parabolic shell.
\newblock \emph{Appl. Numer. Math.}, 150:\penalty0 194--205,
  2020{\natexlab{b}}.

\bibitem[Telega and Lewi\'{n}ski(1998{\natexlab{a}})]{TelegaLewinski1998a}
J.~J. Telega and T.~Lewi\'{n}ski.
\newblock Homogenization of linear elastic shells: {$\Gamma$}-convergence and
  duality. {P}art {I}. {F}ormulation of the problem and the effective model.
\newblock \emph{Bull. Polish Acad. Sci., Techincal Sci.}, 46:\penalty0 1--9,
  1998{\natexlab{a}}.

\bibitem[Telega and Lewi\'{n}ski(1998{\natexlab{b}})]{TelegaLewinski1998b}
J.~J. Telega and T.~Lewi\'{n}ski.
\newblock Homogenization of linear elastic shells: {$\Gamma$}-convergence and
  duality. {P}art {II}. {D}ual homogenization.
\newblock \emph{Bull. Polish Acad. Sci., Techincal Sci.}, 46:\penalty0 11--21,
  1998{\natexlab{b}}.

\bibitem[Zeidler(1988)]{Zeidler1986}
E.~Zeidler.
\newblock \emph{Nonlinear functional analysis and its applications. {IV}}.
\newblock Springer-Verlag, New York, 1988.
\newblock Applications to mathematical physics, Translated from the German and
  with a preface by Juergen Quandt.

\bibitem[Zingaro et~al.(2021)Zingaro, Ded\`{e}, Menghini, and
  Quarteroni]{Quarteroni2021-1}
A.~Zingaro, L.~Ded\`{e}, F.~Menghini, and A.~Quarteroni.
\newblock Hemodynamics of the heart's left atrium based on a {V}ariational
  {M}ultiscale-{LES} numerical method.
\newblock \emph{Eur. J. Mech. B Fluids}, 89:\penalty0 380--400, 2021.

\end{thebibliography}
